\documentclass[]{amsart}

\usepackage{amsfonts}
\usepackage{amscd}
\usepackage{amsmath, mathrsfs, amssymb}
\usepackage{amsthm}
\usepackage{setspace}
\usepackage{hyperref}
\usepackage{color}
\usepackage{epsfig}
\usepackage{here}
\usepackage{graphicx}
\usepackage[all]{xy}
\usepackage{psfrag}
\usepackage{graphicx,transparent}
\usepackage{enumerate}
\usepackage{caption}
 
\theoremstyle{plain}
\newtheorem{theorem}{Theorem}[section]
\newtheorem{lemma}[theorem]{Lemma}
\newtheorem{proposition}[theorem]{Proposition}
\newtheorem{corollary}[theorem]{Corollary}

\theoremstyle{definition}
\newtheorem{remark}[theorem]{Remark}

\newtheorem{example}[theorem]{Example}

\newcommand{\MM}{\mathcal M}

\newcommand{\BM}{\overline{\mathcal M}}

\newcommand{\PP}{\mathbb P}

\newcommand{\calC}{\mathcal C}

\newcommand{\calX}{\mathcal X}

\newcommand{\OO}{\mathcal O}

\newcommand{\RR}{\mathbb R}
\newcommand{\BHH}{\overline{\mathcal H}}
\newcommand{\BPP}{\overline{\mathcal P}}
\newcommand{\calP}{\mathcal P}
\newcommand{\BSS}{\overline{\mathcal S}}
\newcommand{\HH}{\mathcal H}

\newcommand{\calL}{\mathcal L}
\newcommand{\calF}{\mathcal F}

\newcommand{\hyp}{\operatorname{hyp}}
\newcommand{\GL}{\operatorname{GL}}

\newcommand{\even}{\operatorname{even}}
\newcommand{\odd}{\operatorname{odd}}

\newcommand{\wt}{\widetilde}

\newcommand{\Aut}{\operatorname{Aut}}

\newcommand{\ord}{\operatorname{ord}}

\newcommand{\bbC}{\mathbb C}

\newcommand{\bbP}{\mathbb P}

\newcommand{\bbZ}{\mathbb Z}

\newcommand{\SL}{\operatorname{SL}}

\newcommand{\Res}{\operatorname{Res}}

\makeatletter
	\@namedef{subjclassname@2010}{%
	\textup{2010} Mathematics Subject Classification}
	\makeatother

\title{Degenerations of Abelian Differentials}

\author{Dawei Chen}
\address{Department of Mathematics, Boston College, Chestnut Hill, MA 02467}
\email{dawei.chen@bc.edu}


\date{\today}

\thanks{During the preparation of this article the author is partially supported by the NSF CAREER grant DMS-1350396.}

\begin{document}

\begin{abstract}
Consider degenerations of Abelian differentials with prescribed number and multiplicity of zeros and poles. Motivated by the theory of limit linear series, we define twisted canonical divisors on pointed nodal curves to study degenerate differentials, give dimension bounds for their moduli spaces, and establish smoothability criteria. As applications, we show that the spin parity of holomorphic and meromorphic differentials extends to distinguish twisted canonical divisors in the locus of stable pointed curves of pseudocompact type. We also justify whether zeros and poles on general curves in a stratum of differentials can be Weierstrass points. Moreover, we classify twisted canonical divisors on curves with at most two nodes in the minimal stratum in genus three. Our techniques combine algebraic geometry and flat geometry. Their interplay is a main flavor of the paper. 
\end{abstract}

\maketitle

\setcounter{tocdepth}{1}
\tableofcontents

\section{Introduction}
\label{sec:intro}

An \emph{Abelian differential} defines a flat metric on the underlying Riemann surface with conical singularities at its zeros. Varying the flat structure by $\GL^+_2(\RR)$ induces an action on the moduli space of Abelian differentials, called \emph{Teichm\"uller dynamics}. A number of questions about the geometry of a Riemann surface boil down to the study of its $\GL^+_2(\RR)$-orbit, which has provided abundant results in various fields. 
To name a few, Kontsevich and Zorich (\cite{KontsevichZorich}) classified connected components of strata of Abelian differentials with prescribed number and multiplicity of zeros. Surprisingly those strata can have up to three connected components, due to hyperelliptic and spin structures. Eskin and Okounkov (\cite{EskinOkounkov})
used symmetric group representations and modular forms to enumerate special $\GL^+_2(\RR)$-orbits arising from covers of tori with only one branch point, which allows them to compute volume asymptotics of strata of Abelian differentials. Eskin and Masur (\cite{EskinMasur})
proved that the number of families of bounded closed geodesics on generic flat surfaces in a $\GL^+_2(\RR)$-orbit closure has quadratic asymptotics, whose leading term satisfies a formula of Siegel-Veech type. Eskin, Kontsevich, and Zorich (\cite{EKZ})
 further related a version of this Siegel-Veech constant to the sum of Lyapunov exponents under the Teichm\"uller geodesic flow. In joint work with M\"oller (\cite{ChenMoellerAbelian, ChenMoellerQuadratic}) the author applied intersection theory on moduli spaces of curves to prove a nonvarying phenomenon of sums of Lyapunov exponents for Teichm\"uller curves in low genus. A recent breakthrough by Eskin, Mirzakhani, and Mohammadi 
 (\cite{EskinMirzakhani, EskinMirzakhaniMohammadi}) 
 showed that the closure of any $\GL^+_2(\RR)$-orbit is an affine invariant manifold, i.e. locally it is cut out by linear equations of relative period coordinates with real coefficients. More recently Filip (\cite{Filip}) proved that all affine invariant manifolds are algebraic varieties defined over 
 $\overline{\mathbb Q}$, generalizing M\"oller's earlier work on Teichm\"uller curves (\cite{MoellerHodge}). 

Despite the analytic guise in the definition of Teichm\"uller dynamics, there is a fascinating and profound algebro-geometric foundation behind the story, already suggested by some of the results mentioned above. In order to borrow tools from algebraic geometry, the upshot is to understanding degenerations of Abelian differentials, or equivalently, describing a compactification of strata of Abelian differentials, analogous to the 
Deligne-Mumford compactification of the moduli space of curves by adding stable nodal curves. This is the focus of the current paper. 

We use $g$ to denote the genus of a Riemann surface or a smooth, complex algebraic curve. Let $\mu = (m_1, \ldots, m_n)$ be a partition of $2g-2$. Consider the space $\HH(\mu)$ parameterizing pairs 
$(C, \omega)$, where $C$ is a smooth, connected, compact complex curve of genus $g$, and $\omega$ is a holomorphic Abelian differential on $C$ such that $(\omega)_0 = m_1 p_1 + \cdots + m_n p_n$
for distinct points $p_1,\ldots, p_n \in C$. 
We say that $\HH(\mu)$ is the \emph{stratum of (holomorphic) Abelian differentials with signature $\mu$}. For a family of differentials in $\HH(\mu)$, if the underlying smooth curves degenerate to a nodal curve, what is the corresponding limit object of the differentials? In other words, is there a geometrically meaningful compactification of $\HH(\mu)$ and can we describe its boundary elements? 

The space of all Abelian differentials on genus $g$ curves forms a vector bundle $\HH$ of rank $g$, called the 
\emph{Hodge bundle}, over the \emph{moduli space $\MM_g$ of smooth genus $g$ curves}. Let $\BM_g$ be the 
\emph{Deligne-Mumford moduli space of stable nodal genus $g$ curves}. The Hodge bundle $\HH$ extends to a rank $g$ vector bundle 
$\BHH$ over $\BM_g$. If $C$ is nodal, the fiber of $\BHH$ over $C$ can be identified with 
$H^0(C, K)$, where $K$ is the \emph{dualizing line bundle of $C$}. Geometrically speaking, 
$H^0(C, K)$ is the space of \emph{stable differentials $\wt{\omega}$} such that $\wt{\omega}$ has at worst simple pole 
at each node of $C$ with residues at the two branches of every node adding to zero (see e.g. \cite[Chapter 3.A]{HarrisMorrison}). 

Thus it is natural to degenerating Abelian differentials to stable differentials, i.e. compactifying $\HH(\mu)$ in $\BHH$. 
 Denote by $\BHH(\mu)$ the \emph{closure of $\HH(\mu)$ in $\BHH$}. For $(C,\omega) \in \BHH$, let $\wt{C}$ be the \emph{normalization of $C$}. First consider the case when $\omega$ 
has \emph{isolated} zeros and simple poles, i.e. it does not vanish entirely on any irreducible component of $C$. Identify $\omega$ with a stable differential $\wt{\omega}$ on $\wt{C}$. Suppose that 
$$ (\wt{\omega})_0 - (\wt{\omega})_{\infty} = \sum_i a_i z_i + \sum_j (b'_j h'_j + b''_j h''_j) - \sum_k (p'_k + p''_k), $$
where the $z_i$ are the zeros of $\wt{\omega}$ in the smooth locus of $C$, the $h'_j, h''_j$ are the preimages of the node $h_j$ which is not a pole of $\wt{\omega}$, and the $p'_k, p''_{k}$ are the simple poles of $\wt{\omega}$ on the preimages of the node $p_k$. Moreover, $a_i \geq 1$ is the vanishing order of $\wt{\omega}$ at $z_i$, and $b_j', b''_j \geq 0$ are the vanishing orders of $\wt{\omega}$ on $h_j', h_j''$, respectively. Our first result describes which strata closures in $\BHH$ contain such $(C, \omega)$. 

\begin{theorem}
\label{thm:hodge}
In the above setting, we have  
$$(C, \omega) \in \BHH(\cdots, a_i, \cdots, b'_j+1, b''_{j}+1, \cdots). $$
\end{theorem}

Comparing to the signature of $\wt{\omega}$, the notation $(\cdots, a_i, \cdots, b'_j+1, b''_{j}+1, \cdots)$ means keeping all $a_i$ unchanged, adding one to all $b'_j, b''_j$, and getting rid of all $-1$. We remark that when $\omega$ vanishes on a component of $C$, we prove a similar result (see Corollary~\ref{cor:hodge-component}). 

Despite that $\BHH$ has a nice vector bundle structure, a disadvantage of compactifying $\HH(\mu)$ in $\BHH$ is that sometimes it loses information of the limit positions of the zeros of $\omega$, especially if $\omega$ vanishes on a component of the underlying curve. Alternatively, we can consider degenerations in the \emph{Deligne-Mumford moduli 
space $\BM_{g,n}$ of stable genus $g$ curves with $n$ ordered marked points} by marking the zeros of differentials in 
$\HH(\mu)$.
 
For $\mu = (m_1, \ldots, m_n)$ an \emph{ordered} partition of $2g-2$,  
let $\calP(\mu)\subset \MM_{g,n}$ parameterize pointed stable curves 
$(C, z_1, \ldots, z_n)$, where
$ m_1 z_1 + \cdots + m_n z_n $ is a \emph{canonical divisor} on a smooth curve $C$. 
We say that $\calP(\mu)$ is the \emph{stratum of (holomorphic) canonical divisors with signature $\mu$}. 
If we do not order the zeros, then $\calP(\mu) $ is just the projectivization of $\HH(\mu)$, parameterizing differentials modulo scaling.  
Denote by $\BPP(\mu)$ the \emph{closure of $\calP(\mu)$ in $\BM_{g,n}$}. 

Inspired by the theory of limit linear series \cite{EisenbudHarrisLimit}, we focus on nodal curves of the following type in $\BM_{g,n}$. A nodal curve is of \emph{compact type} if every node is \emph{separating}, i.e. removing it makes the whole curve disconnected. A nodal curve is of 
\emph{pseudocompact type} if every node is either separating or a self-intersection point of an irreducible component. We call a node of the latter type 
a \emph{self-node} or an \emph{internal node}, since both have been used in the literature. Note that curves of compact type 
are special cases of pseudocompact type, where all irreducible components are smooth.  

For the reader to get a feel, let us first consider curves of compact type with only one node. Suppose $(C, z_1, \ldots, z_n)\in \BM_{g,n}$ such that 
$C = C_1 \cup_{q} C_2$, where $C_i$ is smooth and has genus $g_i$, and $q$ is a node connecting $C_1$ and $C_2$. In particular, the marked points $z_j$ are different from $q$. Define
$$M_i = \sum_{z_j\in C_i} m_j $$
as the sum of zero orders in each component of $C$. Our next result determines when the stratum closure $\BPP(\mu)$ 
 in $\BM_{g,n}$ contains such $(C, z_1, \ldots, z_n)$. 

\begin{theorem}
\label{thm:canonical}
In the above setting, 
$(C, z_1, \ldots, z_n)\in \BPP(\mu)$ 
if and only if 
$$ \sum_{z_j\in C_i} m_j z_j + (2g_i - 2 - M_i) q \sim K_{C_i} $$
for $i=1,2$, where $\sim$ stands for linear equivalence. 
\end{theorem}

For curves of (pseudo)compact type with more nodes, we prove a more general result (see Theorem~\ref{thm:canonical-more} and Remark~\ref{rem:hodge-pseudo}). In general, we remark that for a pointed nodal curve of pseudocompact type to be contained in $\BPP(\mu)$, the linear equivalence condition as above is necessary, but it may fail to be sufficient (see Example~\ref{ex:dimension-bound} and Proposition~\ref{prop:g=2}). 

For curves of non-pseudocompact type, extra complication comes into play when blowing up a non-separating node and inserting chains of rational curves in order to obtain a \emph{regular} smoothing family of the curve. We explain this issue and discuss a possible solution in Section~\ref{subsec:non-compact}. We also treat certain curves of non-pseudocompact type in low genus by an ad hoc method (see Section~\ref{sec:h(4)}). 

A useful idea is to thinking of the pair 
$$\left(\sum_{z_j\in C_1} m_j z_j + (2g_1 - 2 - M_1) q, \sum_{z_j\in C_2} m_j z_j + (2g_2 - 2 - M_2) q\right)$$ 
appearing in Theorem~\ref{thm:canonical} as a \emph{twisted canonical divisor} (see Section~\ref{subsec:twisted}), in the sense that each entry is an ordinary canonical divisor on $C_i$. Note that if $2g_i - 2 - M_i < 0$, then it is not effective, i.e. the corresponding differential on $C_i$ is meromorphic with a pole. In general, we call such $C_i$ a \emph{polar component}. Conversely if on $C_i$ a twisted canonical divisor is effective, we call it 
a \emph{holomorphic component}. 

Therefore, it is nature to enlarge our study by considering \emph{meromorphic differentials} and their degenerations, also for the sake of completeness. Take a sequence 
of integers $\mu = (k_1, \ldots, k_r, - l_1, \ldots, -l_s)$ such that $k_i, l_j > 0$, and 
$$\sum_{i=1}^r k_i - \sum_{j=1}^s l_{j} = 2g-2.$$ 
We still use $\HH(\mu)$ to denote the \emph{stratum of meromorphic differentials with signature $\mu$}, parameterizing meromorphic differentials 
$\omega$ on connected, closed genus $g$ Riemann surfaces $C$ such that 
$$(\omega)_0 - (\omega)_{\infty} = \sum_{i=1}^r k_i z_i - \sum_{j=1}^s l_{j} p_j$$ 
for distinct $z_i, p_j\in C$. We sometimes allow the case $k_i = 0$ by treating $z_i$ as a marked point irrelevant to the differential. Let $\calP(\mu)$ be the corresponding \emph{stratum of meromorphic canonical divisors with signature $\mu$}. As in the case of holomorphic differentials, 
ordering and marking the zeros and poles, we denote by $\BPP(\mu)$ the \emph{closure of $\calP(\mu)$ in $\BM_{g,n}$} with $n = r + s$. As an analogue of Theorem~\ref{thm:canonical}, we have the following result.  

\begin{theorem}
\label{thm:main-meromoprhic-one-node}
Suppose $C = (C_1\cup_q C_2, z_1, \ldots, z_r, p_1, \ldots, p_s) \in \BM_{g,n}$ is a curve of compact type with one node $q$ such that 
$C_i$ has genus $g_i$ and both $C_i$ are polar components. Then $C\in \BPP(k_1, \ldots, k_r, - l_1, \ldots, -l_s)$ if and only if 
$$ \sum_{z_j\in C_i} k_j z_j - \sum_{p_h \in C_i} l_h p_h + (2g_i - 2 - M_i) q \sim K_{C_i} $$
for $i = 1, 2$, where $M_i = \sum_{z_j\in C_i} k_j - \sum_{p_h \in C_i} l_h$. 
\end{theorem}

Again, here we treat the pair $\sum_{z_j\in C_i} k_j z_j - \sum_{p_h \in C_i} l_h p_h + (2g_i - 2 - M_i)q$ as a \emph{twisted meromorphic 
canonical divisor} on $C$. For curves of (pseudo)compact type with more nodes, we prove a more general result for twisted meromorphic canonical divisors 
 (see Theorem~\ref{thm:twisted-meromorphic}). We remark that in both holomorphic and meromorphic cases, the upshot of our proof is to establishing certain dimension bounds for irreducible components of moduli spaces of twisted canonical divisors (see Section~\ref{subsec:dimension}). 

Note that for special signatures $\mu$, $\calP(\mu)$ can be \emph{disconnected}. Kontsevich and Zorich (\cite{KontsevichZorich}) classified connected components for strata of holomorphic differentials. In general, $\calP(\mu)$ may have \emph{up to three} connected components, distinguished by \emph{hyperelliptic, odd} or \emph{even spin structures}. When these components exist, we adapt the same notation as \cite{KontsevichZorich}, using  
``$\hyp$'', ``$\odd$'' and ``$\even$'' to distinguish them. Recently Boissy (\cite{Boissy}) classified connected components for strata of meromorphic differentials, which are similarly distinguished by hyperelliptic and spin structures. Therefore, when $\calP(\mu)$ has more than one connected component, one can naturally ask how to distinguish the boundary points in the closures of its connected components.  

For hyperelliptic components, it is well-known that a degenerate hyperelliptic curve in $\BM_{g}$ can be described explicitly using the theory of \emph{admissible covers} (\cite{HarrisMumford}), by comparing to the moduli space of stable genus zero curves with $2g+2$ marked points, 
where the marked points correspond to the $2g+2$ branch points of a hyperelliptic double cover. In this way we have a good understanding of compactifications of hyperelliptic components. For spin components, the following result distinguishes their boundary points in the locus of curves of pseudocompact type. 

\begin{theorem}
\label{thm:spin}
Let $\calP(\mu)$ be a stratum of holomorphic or meromorphic differentials with signature $\mu$ that possesses two spin components 
$\calP(\mu)^{\odd}$ and $\calP(\mu)^{\even}$. Then $\BPP(\mu)^{\odd}$ and $\BPP(\mu)^{\even}$ are disjoint in the locus of curves of pseudocompact type. 
\end{theorem}

However, we remark that in the locus of curves of non-pseudocompact type in $\BM_{g,n}$, these components can intersect (see Theorem~\ref{thm:double-conic}). 

For a point $p$ on a genus $g$ Riemann surface $C$, if $h^0(C, gp) \geq 2$, we say that $p$ is a \emph{Weierstrass point}. The study of Weierstrass points has been a rich source for understanding the geometry of Riemann surfaces (see e.g. \cite[Chapter I, Exercises E]{ACGH}). In the context of strata of holomorphic differentials, for example, if $m_1 z_1 + \cdots + m_n z_n$ is a canonical divisor of $C$ such that $m_1 \geq g$, then it is easy to see that $z_1$ is a Weierstrass point. Furthermore, the Weierstrass gap sequences of the unique zero of general differentials in the minimal strata $\calP(2g-2)$ were calculated by Bullock (\cite{Bullock}). Using techniques developed in this paper, we can prove the following result. 

\begin{theorem}
\label{thm:weierstrass}
Let $(C, z_1, \ldots, z_r, p_1, \ldots, p_s)$ be a general curve parameterized in (the non-hyperelliptic components of) 
$\calP(k_1, \ldots, k_r, -l_1, \ldots, -l_s)$. Then $z_i$ is not a Weierstrass point. 
\end{theorem}

We also establish similar results as above in a number of other cases (see Propositions~\ref{prop:weierstrass-holomorphic} and~\ref{prop:weierstrass-meromorphic}). 

This paper is organized as follows. In Section~\ref{sec:preparation}, we introduce basic tools that are necessary to prove our results. In Section~\ref{sec:hodge}, we consider degenerations of Abelian differentials 
in the Hodge bundle $\BHH$ and prove Theorem~\ref{thm:hodge}. In Section~\ref{sec:canonical}, we consider degenerations 
of canonical divisors in $\BM_{g,n}$ and prove Theorems~\ref{thm:canonical} and~\ref{thm:main-meromoprhic-one-node}. In 
Section~\ref{sec:spin}, we consider boundary points of connected components of $\HH(\mu)$ and prove Theorem~\ref{thm:spin}. In Section~\ref{sec:weierstrass}, we study Weierstrass point behavior for general differentials in $\HH(\mu)$ and prove Theorem~\ref{thm:weierstrass}. 
Finally in Section~\ref{sec:h(4)}, we carry out a case study by analyzing the boundary of $\BPP(4)$ in $\BM_{3,1}$ in detail. 

Our techniques combine both algebraic geometry and flat geometry. The interplay between the two fields is a main flavor throughout the paper. For that reason, we will often identify smooth, complex algebraic curves with Riemann surfaces and switch our language back and forth.  

{\bf Acknowledgements.} The author is grateful to Madhav Nori and Anand Patel for many stimulating discussions. The author also wants to thank Matt Bainbridge, Gabriel Bujokas, Izzet Coskun, Alex Eskin, Simion Filip, Sam Grushevsky, Joe Harris, Yongnam Lee, Martin M\"oller, Nicola Tarasca, and Anton Zorich for relevant conversations and their interests in this work. Quentin Gendron informed the author that he has obtained some of the results in Sections~\ref{sec:spin} and~\ref{sec:h(4)} independently (\cite{Gendron}), the methods being in some cases related, in some cases disjoint, and the author thanks him for communications and comments on an earlier draft of this work. Results in this paper were announced at the conference ``Hyperbolicity in Algebraic Geometry'', Ilhabela, January 2015. The author thanks the organizers Sasha Anan'in, Ivan Cheltsov, and Carlos Grossi for their invitation and hospitality.

\section{Preliminaries}
\label{sec:preparation}

In this section, we review basic background material and introduce necessary techniques that will be used later in the paper. 

\subsection{Abelian differentials and translation surfaces}
\label{subsec:translation}

A \emph{translation surface} (also called a \emph{flat surface}) is a closed, topological surface $C$ together with a finite set 
$\Sigma\subset C$ such that: 
\begin{itemize}
\item 
There is an atlas of charts from $C\backslash \Sigma \to \bbC$ with transition functions given by translation. 
\item 
For each $p\in \Sigma$, under the Euclidean metric of $\bbC$ 
 the total angle at $p$ is $(2\pi) \cdot k$ for some $k\in \bbZ^+$. 
\end{itemize}
We say that $p$ is a \emph{saddle point of cone angle} $(2\pi)\cdot k$. 

Equivalently, a translation surface is a closed Riemann surface $C$ with a holomorphic Abelian differential $\omega$, not identically zero: 
\begin{itemize}
\item 
The set of zeros of $\omega$ corresponds to $\Sigma$ in the first definition. 
\item 
If $p$ is a zero of $\omega$ of order $m$, then the cone angle at $p$ is $ (2\pi)\cdot (m+1) $. 
\end{itemize}

Let us briefly explain the equivalence between translation surfaces and Abelian differentials. Given a translation surface, away from its saddle points differentiating the local coordinates yields a globally defined holomorphic differential. Conversely, integrating an Abelian differential away from its zeros provides an atlas of charts with transition functions given by translation. Moreover, a saddle point $p$ has cone angle $(2\pi)\cdot (m+1)$ if and only if locally $\omega = d(z^{m+1})\sim z^m dz$ for a suitable coordinate $z$, hence if and only if $p$ is a zero of $\omega$ of order $m$. We refer to \cite{Zorich} for a comprehensive introduction to translation surfaces. 

\subsection{Strata of Abelian differentials and canonical divisors}
\label{subsec:strata}

Take a sequence of positive integers $\mu = (m_1, \ldots, m_n)$ such that $\sum_{i=1}^n m_i = 2g-2$. We say that 
$\mu$ is a \emph{partition} of $2g-2$. Define 
$$ \HH(\mu) = \Big\{ (C, \omega) \mid C \ \mbox{is a closed, connected Riemann surface of genus}\ g, $$
$$\omega \ \mbox{is an Abelian differential on}\ C \ \mbox{such that}\ (\omega)_0 = m_1 p_1 + \cdots + m_n p_n \Big\}. $$
We say that $\HH(\mu)$ is the \emph{stratum of (holomorphic) Abelian differentials with signature $\mu$}. 
Using the description in Section~\ref{subsec:translation}, equivalently $\HH(\mu)$ parameterizes translation surfaces with $n$ saddle points, each having cone angle 
$(m_i+1) \cdot (2\pi)$. By using relative period coordinates (see e.g. \cite[Section 3.3]{Zorich}), $\HH(\mu)$ can be regarded as a complex orbifold  
of dimension 
$$\dim_{\bbC} \HH(\mu) = 2g + n -1, $$
where $n$ is the number of entries in $\mu$. 

For special partitions $\mu$, $\HH(\mu)$ can be \emph{disconnected}. Kontsevich and Zorich 
(\cite[Theorems 1 and 2]{KontsevichZorich}) classified connected components of $\HH(\mu)$ for all $\mu$. 
If a translation surface $(C, \omega)$ has $C$ being hyperelliptic, $(\omega)_0 = (2g-2)z$ or $(\omega)_0 = (g-1)(z_1+z_2)$, where 
$z$ is a Weierstrass point of $C$ in the former or $z_1$ and $z_2$ are conjugate under the hyperelliptic involution of $C$ in the latter, we say that $(C, \omega)$ is a \emph{hyperelliptic translation surface}. Note that being a hyperelliptic translation surface not only requires $C$ to be hyperelliptic, but also imposes extra conditions to $\omega$ (see \cite[Definition 2 and Remark 3]{KontsevichZorich}).  

In addition, for a nonhyperelliptic translation surface $(C, \omega)$, if 
$(\omega)_0 = 2 k_1 z_1 + \cdots + 2k_n z_n$, then the line bundle 
$$\OO_C\left(\sum_{i=1}^n k_i z_i\right)$$ 
is a square root of $K_C$, which is called a \emph{theta characteristic}. Such a theta characteristic along with its \emph{parity}, i.e. 
$$h^0\left(C, \sum_{i=1}^n k_i z_i\right) \pmod{2}$$ 
is called a \emph{spin structure}. In general, $\HH(\mu)$ may have \emph{up to three} connected components, distinguished by possible hyperelliptic and spin structures. 

Note that two Abelian differentials are multiples of each other if and only if their associated zero divisors are the same. 
Therefore, it makes sense to define the \emph{stratum of canonical divisors with signature $\mu$} in $\MM_{g,n}$, denoted by 
$\calP(\mu)$, parameterizing $(C, z_1, \ldots, z_n)$ such that $\sum_{i=1}^n m_i z_i$ is a canonical divisor in $C$. Here we choose to \emph{order} the zeros only for the convenience of stating related results. Alternatively if one considers the corresponding stratum of canonical divisors without ordering the zeros, it is just the projectivization of $\HH(\mu)$. In particular, 
$$ \dim_{\bbC} \calP(\mu) = \dim_{\bbC} \HH(\mu)-1 = 2g + n -2. $$

\subsection{Meromorphic differentials and translation surfaces with poles}
\label{subsec:pole}

One can also consider the flat geometry associated to meromorphic differentials on Riemann surfaces. In this case we obtain 
flat surfaces with infinite area, called \emph{translation surfaces with poles}. 

For $k_1, \ldots, k_r, l_1, \ldots, l_s \in \bbZ^{+}$ 
such that $\sum_{i=1}^r k_i - \sum_{j=1}^s l_j = 2g-2$, denote by 
$$\HH(k_1, \ldots, k_r, -l_1, \ldots, -l_s)$$ the 
\emph{stratum of meromorphic differentials} parameterizing $(C, \omega)$, where $\omega$ is a meromorphic differential on a closed, connected genus $g$ Riemann surface $C$ such that $\omega$ has zeros of order $k_1, \ldots, k_r$ and poles of order 
$l_1, \ldots, l_s$, respectively. The dimension and connected components of $\HH(k_1, \ldots, k_r, -l_1, \ldots, -l_s)$ have been determined 
by Bossy (\cite[Theorems 1.1, 1.2 and Lemma 3.5]{Boissy}), using an \emph{infinite zippered rectangle construction}. In particular, 
if $s > 0$, i.e. if there is at least one pole, then 
$$ \dim_{\bbC}\HH(k_1, \ldots, k_r, -l_1, \ldots, -l_s) = 2g - 2 + r + s. $$
If we consider meromorphic differentials modulo scaling, i.e. \emph{meromorphic canonical divisors}, then the corresponding stratum has dimension  
$$ \dim_{\bbC}\calP(k_1, \ldots, k_r, -l_1, \ldots, -l_s) = 2g - 3 + r + s. $$
As in the case of holomorphic Abelian differentials,  
$\HH(k_1, \ldots, k_r, -l_1, \ldots, -l_s)$ can be disconnected due to possible hyperelliptic and spin structures (\cite[Section 5]{Boissy}), but all connected components of a stratum have the same dimension. 

A special case is when $\omega$ has a simple pole at $p$. Under flat geometry, the local neighborhood of $p$ can be visualized as a \emph{half-infinite cylinder} (see \cite[Figure 3]{Boissy}). The \emph{width} of the cylinder corresponds to the \emph{residue} of $\omega$ at $p$. 

For a pole of order $m \geq 2$, one can glue $2m-2$ \emph{basic domains} appropriately to form a flat-geometric presentation (see \cite[Section 3.3]{Boissy}). Each basic domain is a ``broken half-plane'' whose boundary consists of a half-line to the left and a paralell half-line to the right, connected by finitely many broken line segments. In particular, the \emph{residue} of a pole can be read off from the complex lengths of the broken line segments and the gluing pattern. 

For example, for $k \geq 0$ the differential $z^k dz$ gives a zero of order $k$, so locally one can glue $2k+2$ half-disks consecutively to form a cone of angle $2\pi \cdot (k+1)$, see Figure~\ref{fig:zero-k}. 
\begin{figure}[h]
    \centering
    \psfrag{B1}{$B_1$}
    \psfrag{B2}{$B_2$}
    \psfrag{Bk+1}{$B_{k+1}$}
    \psfrag{A1}{$A_1$}
    \psfrag{A2}{$A_2$}
    \psfrag{A3}{$A_3$}
    \psfrag{Ak+1}{$A_{k+1}$}   
    \includegraphics[scale=1.2]{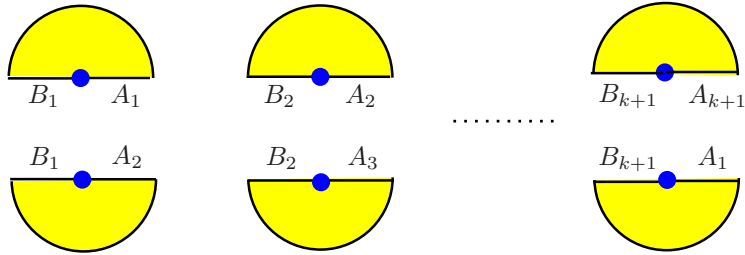}
    \caption{\label{fig:zero-k} A zero of order $k$}
    \end{figure}
    
Now let $w = 1/z$, and the differential with respect to $w$ has a pole of degree $k+2$ with \emph{zero} residue. In terms of the flat-geometric language, the $2k+2$ half-disks transform to $2k+2$ half-planes (with the disks removed), where the newborn left and right half-line boundaries are identified in pairs by the same gluing pattern, see Figure~\ref{fig:pole-k}. 
\begin{figure}[h]
    \centering
    \psfrag{Bi}{$B_i$}
    \psfrag{Ai}{$A_i$}
    \psfrag{Aj}{$A_{i+1}$}   
    \includegraphics[scale=0.8]{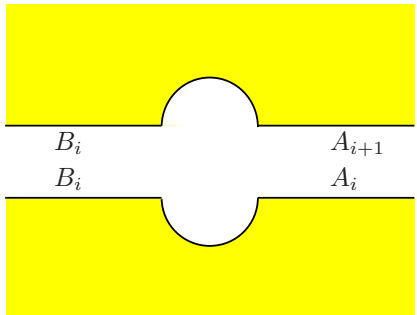}
    \caption{\label{fig:pole-k} Half-planes transformed from half-disks}
    \end{figure}
    
Furthermore, varying the positions of the half-line boundaries with suitable rotating and scaling can produce poles of order $k+2$ with arbitrary nonzero residues (see \cite[Section 2.2]{Boissy}). 

\subsection{Deligne-Mumford stable curves and stable one-forms}
\label{subsec:stable}

Let $\BM_{g,n}$ be the \emph{Deligne-Mumford moduli space of stable nodal genus $g$ curves with $n$ ordered marked points 
$(C, p_1, \ldots, p_n)$}. The stability condition means that $\Aut (C, p_1, \ldots, p_n)$ is finite, or equivalently, the normalization of every rational component of $C$ contains at least three special points (preimages of a node or marked points). For 
$S\subset \{1,\ldots,n \}$, denote by $\Delta_{i; S}$ the boundary component of $\BM_{g,n}$ whose general point parameterizes two smooth curves of genus $i$ and $g-i$, respectively, glued at a node such that the genus $i$ component only contains 
the marked points labeled by $S$ in the smooth locus. For $i = 0$ (resp. $i = g$), we require that $|S| \geq 2$ (resp. $|S|\leq n-2$) to fulfill the stability condition. The codimension of $\Delta_{i; S}$ in $\BM_{g,n}$ is one, so we call it a 
\emph{boundary divisor}. 

The \emph{Hodge bundle} $\BHH$ is a rank $g$ vector bundle on $\BM_{g}$ (in the orbifold sense). Formally it is defined as 
$$ \BHH := \pi_{*} \omega_{\calC/\BM_g}, $$
where $\pi: \calC\to \BM_g$ is the universal curve and $\omega_{\calC/\BM_g}$ is the \emph{relative dualizing line bundle} of $\pi$. Geometrically speaking, the fiber of $\BHH$ over $C$ is $H^0(C, K)$, where $K$ is the \emph{dualizing line bundle} of $C$. If $C$ is nodal, then $H^0(C, K)$ 
can be identified with the space of stable differentials on the normalization $\wt{C}$ of $C$. A \emph{stable differential} 
$\wt{\omega}$ on $\wt{C}$ is a meromorphic differential that is holomorphic away from preimages of nodes of $C$ and has at worst simple pole at the preimages of a node, with residues on the two branches of a polar node adding to zero (see e.g. \cite[Chapter 3.A]{HarrisMorrison}).

\subsection{Admissible covers}
\label{subsec:admissible}

Harris and Mumford (\cite{HarrisMumford}) developed the theory of admissible covers to deal with degenerations of branched covers of smooth curves to covers of nodal curves. Let $f: C\to D$ be a finite morphism of nodal curves satisfying the following conditions: 
\begin{itemize}
\item 
$f$ maps the smooth locus of $C$ to the smooth locus of $D$ and maps the nodes of $C$ to the nodes of $D$. 
\item 
Suppose $f(p) = q$ for a node $p\in C$ and a node $q\in D$. Then there exist suitable local coordinates 
$x, y$ for the two branches at $p$, and local coordinates  
$u, v$ for the two branches at $q$, such that 
$$u = f(x) = x^m, \quad v = f(y) = y^m$$ 
for some $m \in \bbZ^{+}$, see Figure~\ref{fig:cover-local}. 
\begin{figure}[h]
    \centering
    \psfrag{D}{$D$}
    \psfrag{C}{$C$}
    \psfrag{p}{$p$}
       \psfrag{q}{$q$}
        \psfrag{m}{$m$ sheets} 
    \includegraphics[scale=0.8]{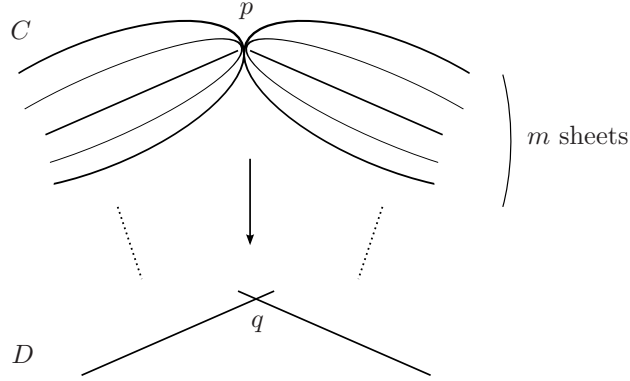}
    \caption{\label{fig:cover-local} An admissible cover with a node of order $m$}
    \end{figure}  
\end{itemize} 
We say that such a map $f$ is an \emph{admissible cover}. 
The reader can refer to \cite[Chapter 3.G]{HarrisMorrison} for a comprehensive introduction to admissible covers. In this paper we will only use admissible \emph{double} covers of rational curves as degenerations of hyperelliptic coverings of $\bbP^1$. In particular, the closure of the locus of hyperelliptic curves in $\BM_g$ is isomorphic to the moduli space $\BM_{0,2g+2}/\mathfrak{S}_{2g+2}$ of stable genus zero curves 
with $2g+2$ \emph{unordered} marked points. 

\subsection{Limit linear series}
\label{subsec:lls}

A \emph{linear series} $g^r_d$ on a smooth curve $C$ consists of a degree $d$ line bundle $L$ with a subspace 
$V\subset H^0(C, L)$ such that $\dim V = r+1$. For a point $z\in C$, take a basis $\sigma_0, \ldots, \sigma_r$ of $V$ such that 
the vanishing orders $a_i = \ord_z(\sigma_i)$ are strictly increasing. We say that $0\leq a_0< \cdots < a_r$ is the \emph{vanishing sequence} of $(L, V)$ at $z$, which is apparently independent of the choices of a basis. Set $\alpha_i = a_i - i$. The sequence $0\leq \alpha_0 \leq \cdots \leq \alpha_r$ is called the \emph{ramification sequence} of $(L, V)$. 

Now consider a nodal curve $C$. Recall that if removing any node makes the whole curve disconnected, $C$ is called of \emph{compact type}. Equivalently, a nodal curve is of compact type if and only if its Jacobian is compact, which is then isomorphic to the product of Jacobians of its connected components. One more equivalent definition uses the \emph{dual graph} of a nodal curve, whose vertices correspond to components of the curve and 
 two vertices are linked by an edge if and only if the corresponding two components intersect at a node. It is easy to see that a curve is of compact type if and only if its dual graph is a tree. 

Eisenbud and Harris (\cite{EisenbudHarrisLimit}) established a theory of \emph{limit linear series} as a powerful tool to study degenerations of linear series from smooth curves to curves of compact type. If $C$ is a curve of compact type with irreducible components $C_1, \ldots, C_k$, 
a (refined) limit linear $g^r_d$ is a collection of ordinary $g^r_d$'s $(L_i, V_i)$ on each $C_i$ such that if $C_i$ and $C_j$ intersect at a node $q$ and if $(a_0, \ldots, a_r)$ and $(b_0, \ldots, b_r)$ are the vanishing sequences of $(L_i, V_i)$ and $(L_j, V_j)$ at $q$, respectively, then 
$a_l + b_{r-l} = d$ for all $l$. 

Eisenbud and Harris showed that if a family of $g^r_d$'s on smooth curves degenerate to a curve of compact type, then the limit object is a limit linear $g^r_d$. Furthermore, they constructed a limit linear series moduli scheme $G^r_d$ that is compatible with imposing ramification conditions 
to points in the smooth locus of a curve, came up with a lower bound for any irreducible component of $G^r_d$, and used it to study smoothability of limit linear series. They also remarked that the method works for a larger class of curves, called \emph{tree-like curves}, which we call of \emph{pseudocompact type} in our context. Recall that a curve is of pseudocompact type, if every node is either separating or a self-node, i.e. arising from the self-intersection of an irreducible component of the curve. Equivalently, a curve is of pseudocompact type if any closed path in its dual graph is a loop connecting a vertex to itself. 

We want to apply limit linear series to the situation when canonical divisors with $n$ distinct zeros with prescribed vanishing orders
degenerate in the Deligne-Mumford moduli space $\BM_{g,n}$. In this context we need to treat the case of limit canonical series $g^{g-1}_{2g-2}$, because on a smooth genus $g$ curve a $g^{g-1}_{2g-2}$ is uniquely given by the canonical line bundle along with the space of holomorphic Abelian differentials. We illustrate its application in some cases (see Example~\ref{ex:dimension-bound} and Proposition~\ref{prop:g=2}). Nevertheless, in general our situation is slightly different, since an element in a stratum of differentials is a single section of the canonical line bundle, not the whole space of sections. In principle keeping track of degenerations of $g^{g-1}_{2g-2}$ along with a special section could provide finer information, but in practice it seems complicated to work with. Instead, in Section~\ref{subsec:twisted} we introduce the notion of \emph{twisted canonical divisors} that play the role of ``limit canonical divisors'' on curves of pseudocompact type. We also discuss a possible extension of twisted canonical divisors to curves of non-pseudocompact type in Section~\ref{subsec:non-compact}. 

\subsection{Moduli of spin structures}
\label{subsec:spin}

Recall that a \emph{theta characteristic} is a line bundle $L$ on a smooth curve $C$ such that $L^{\otimes 2} = K_C$, i.e. $L$ is a square root of the canonical line bundle. A theta characteristic is also called a \emph{spin structure}, whose \emph{parity} is given by $h^0(C, L) \pmod{2}$. In particular, a spin structure is either \emph{even} or \emph{odd}, and the parity is deformation invariant (see \cite{Atiyah, Mumford}). Cornalba (\cite{Cornalba}) constructed a \emph{compactified moduli space of spin curves} $\BSS_g = \BSS_g^{+}\sqcup \BSS_g^{-}$ over $\BM_g$, which defines limit spin structures and further distinguishes odd and even parities. 

Let us first consider spin structures on curves of compact type. Take a nodal curve $C$ with two smooth components $C_1$ and $C_2$ union at a node $q$. Blow up $q$ to insert a $\bbP^1$ between $C_1$ and $C_2$ with new nodes $q_i= C_i\cap \bbP^1$ for $i=1,2$. Such $\bbP^1$ is called 
an \emph{exceptional component}. 
Then a spin structure $\eta$ on $C$ consists of the data 
$$ (\eta_1, \eta_2, \OO(1)), $$ 
where $\eta_i$ is an ordinary theta characteristic on $C_i$ and $\OO(1)$ is a line bundle of degree one on the exceptional component. Note that the total degree 
of $\eta$ is 
$$(g_1 - 1) + (g_2 - 1) + 1 = g-1,$$
which remains to be one half of the degree of $K_C$. 
Since $h^0(\bbP^1, \OO(1)) = 2$, the parity of $\eta$ is determined by 
$$h^0(C_1, \eta_1) + h^0(C_2, \eta_2) \pmod{2}.$$ 
In other words, $\eta$ is even (resp. odd) if and only if $\eta_1$ and $\eta_2$ have the same (resp. opposite) parity. If there is no confusion, we will simply drop the exceptional component and treat $(\eta_1, \eta_2)$ as a limit theta characteristic. The same description works for spin structures on a curve of compact type with more nodes, by inserting an exceptional $\bbP^1$ between any two adjacent components, and the parity is determined by the sum of the parities on each non-exceptional component. 

If $C$ is a nodal curve of non-compact type, say, by identifying $q_1, q_2 \in \wt{C}$ to form a non-separating node $q$, there are \emph{two} kinds of spin structures on $C$. The \emph{first kinds} are just square roots of $K_{C}$, which can be obtained as follows. Take a line bundle $L$ on $\wt{C}$ such that $L^{\otimes 2} \cong \wt{C}(q_1 + q_2)$. For each parity, there is precisely one way to identify the fibers of $L$ over $q_1$ and $q_2$, such that it descends to a square root of $K_C$ with the desired parity. The \emph{second kinds} are obtained by blowing up $q$ to insert a $\bbP^1$ attached to $\wt{C}$ at $q_1$ and $q_2$, and suitably gluing a theta characteristic $L$ on $\wt{C}$ to $\OO(1)$ on the exceptional component. In this case the parity is the same as that of $L$. 

\section{Degenerations in the Hodge bundle}
\label{sec:hodge}

In this section we consider degenerations of holomorphic Abelian differentials in the Hodge bundle $\BHH$ over $\BM_g$. Let us first prove Theorem~\ref{thm:hodge}. Recall that $\wt{C}$ is the normalization of $C$. Identify $\omega$ with a stable differential $\wt{\omega}$ on $\wt{C}$ satisfying that 
$$ (\wt{\omega})_0 - (\wt{\omega})_{\infty} = \sum_i a_i z_i + \sum_j (b'_j h'_j + b''_j h''_j) - \sum_k (p'_k + p''_k), $$
where the $z_i$ are the zeros of $\wt{\omega}$ in the smooth locus of $C$, the $h'_j, h''_j$ are the preimages of the node $h_j$ which is not a pole of $\wt{\omega}$, and the $p'_k, p''_{k}$ are the simple poles of $\wt{\omega}$ on the preimages of the node $p_k$, see Figure~\ref{fig:nodal-zhp}.  
 \begin{figure}[h]
    \centering
    \psfrag{z}{$z$}
    \psfrag{p}{$p$}
    \psfrag{h}{$h$}
    \includegraphics[scale=0.8]{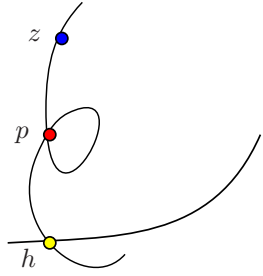}
    \caption{\label{fig:nodal-zhp} A nodal curve with zeros and holomorphic and polar nodes}
    \end{figure}

Moreover, $a_i \geq 1$ is the vanishing order of $\wt{\omega}$ at $z_i$, and $b_j', b''_j \geq 0$ are the vanishing orders of $\wt{\omega}$ on $h_j', h_j''$, respectively. Then Theorem~\ref{thm:hodge} states that $(C, \omega)$ is contained in the closure of $\HH(\cdots, a_i, \cdots, b'_j+1, b''_{j}+1, \cdots)$ 
in the Hodge bundle over $\BM_g$. 

\begin{proof}[Proof of Theorem~\ref{thm:hodge}]
We will carry out two local operations. First, we need to smooth out a holomorphic node $h$ with zero order $b'$ and $b''$ on the two branches of $h$ 
 to two smooth points of zero order $b'+1$ and $b''+1$, respectively.  Secondly, we need to smooth out simple poles. 

Let us describe the first operation. Recall the notation that the preimages of $h$ in the normalization $\wt{C}$ 
 are $h'$ and $h''$. In $\wt{C}$, take two sufficiently small parallel intervals of equal length to connect $h'$ to a nearby point $q''$ and connect $h''$ to a nearby point $q'$ in reverse directions, cut along the intervals, and finally identify the edges by translation as in Figure~\ref{fig:interval}. 
 \begin{figure}[h]
    \centering
    \psfrag{a}{$h'$}
    \psfrag{d}{$h''$}
    \psfrag{b}{$q'$}
    \psfrag{c}{$q''$}
    \psfrag{A}{$A$}
    \psfrag{B}{$B$}   
    \includegraphics[scale=0.15]{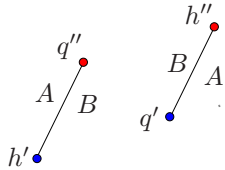}
    \caption{\label{fig:interval} Two parallel interval slits}
    \end{figure}
    
Locally we obtain two new zeros $h' = q'$ and $h'' = q''$ of order $b'+1$ and $b''+1$, respectively. The zero orders increase by one 
for each, because the cone angles at $q'$ and at $q''$ are both $2\pi$, so after this operation the new cone angle at each zero gain an extra $2\pi$. In particular, 
as long as the interval is small enough but nonzero, it gives rise to a differential on a genus $g$ Riemann surface, which preserves the other zero and pole orders of $\wt{\omega}$.  Now shrinking the interval to a point, this operation amounts to identifying $h'$ and $h''$, thus recovering the 
stable differential $(C, \omega)$. 

Next,  let $p$ be a simple pole with preimages $p'$ and $p''$ in $\wt{C}$. As mentioned in Section~\ref{subsec:pole}, the local flat geometry of 
$\wt{\omega}$ at $p'$ and $p''$ is presented by two half-infinite cylinders 
 with $p' = +\infty$ and $p'' = -\infty$, see Figure~\ref{fig:half-cylinder}. 
 \begin{figure}[h]
    \centering
    \psfrag{a}{$a$}
    \psfrag{b}{$b$}
    \psfrag{p}{$p'$}
    \psfrag{q}{$p''$}
    \includegraphics[scale=0.2]{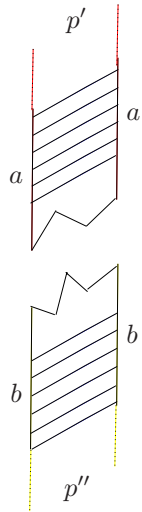}
    \caption{\label{fig:half-cylinder} Half-infinite cylinders around simple poles}
    \end{figure}

The condition $\Res_{p'}(\wt{\omega}) + \Res_{p''}(\wt{\omega}) = 0$ implies that both cylinders have the same width (in opposite direction). Truncate the half-infinite cylinders by two parallel vectors (given by the residues) and identify the top and bottom by translation as in Figure~\ref{fig:plumbing}. 
 \begin{figure}[h]
    \centering
    \psfrag{a}{$a$}
    \psfrag{b}{$b$}
    \psfrag{c}{$c$}
    \includegraphics[scale=0.2]{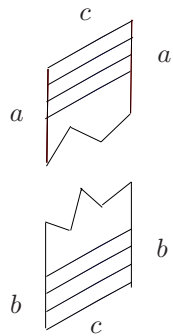}
    \caption{\label{fig:plumbing} Local view of plumbing a cylinder}
\end{figure}

The cylinders become of finite length, i.e., locally the simple pole disappears. This operation is called \emph{plumbing a cylinder} in the literature, see Figure~\ref{fig:plumbing-global}. 
 \begin{figure}[h]
    \centering
    \includegraphics[scale=0.5]{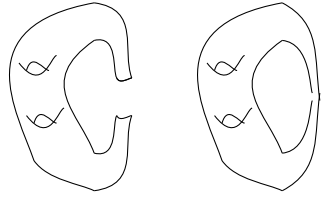}
    \caption{\label{fig:plumbing-global} Global view of plumbing a cylinder}
\end{figure}
In particular, the plumbing operation does not produce any new zeros nor poles. Conversely, extending the two finite cylinders to infinity on both ends, we recover the pair of simple poles. The reader can refer to \cite[Section 6.3]{Wolpert} for an explicit example of analytically plumbing an Abelian differential at a simple pole.

Now carrying out the two operations locally for all holomorphic nodes and simple poles one by one, we thus conclude that 
the stable differential $(C, \omega)$ can be realized as a degeneration of holomorphic differentials in the desired stratum. 
\end{proof}

\begin{corollary}
\label{cor:node-pole}
Let $\omega \in H^0(C, K)$ on a nodal curve $C$. If $(\omega)_0 = m_1 z_1 + \cdots + m_n z_n$ such that every $z_i$ is in the smooth locus of $C$, then 
$$(C, \omega) \in \BHH(m_1, \ldots, m_n).$$  
\end{corollary}

\begin{proof}
We first remark that by assumption $C$ cannot have separating nodes. Otherwise if $X$ was a connected component of $C$ separated by such a node $q$, 
since the restriction of $K_C$ to $X$ is $K_X(q)$, it would have a base point at $q$, contradicting that $\omega$ has no zero in the nodal locus of $C$. Now let us proceed with the proof of the corollary. Identify $\omega$ with a stable differential $\wt{\omega}$ on the normalization of $C$. 
By assumption, $\wt{\omega}$ has simple poles at preimages of each node of $C$, and hence there is no holomorphic node. The desired result thus follows as a special case of Theorem~\ref{thm:hodge}.  
\end{proof}

Nori (\cite{Nori}) informed the author that the above corollary can also be proved by studying first-order deformations of such $(C, \omega)$. 

\begin{example}
Let $C$ consist of two elliptic curves $C'$ and $C''$ meeting at two nodes $p_1$ and $p_2$. Let $p'_i \in C'$ and 
$p''_i\in C''$ be the preimages of $p_i$ in the normalization of $C$. Let $\omega$ be a section of 
$K_C$ such that $\wt{\omega}|_{C'}$ has a double zero at a smooth point $z_1$ and two simple poles at $p'_1, p'_2$ and 
$\wt{\omega}|_{C''}$ has a double zero at a smooth point $z_2$ and two simple poles at $p''_1, p''_2$, see Figure~\ref{fig:banana-ee}. 

 \begin{figure}[h]
    \centering
    \psfrag{z1}{$z_1$}
    \psfrag{z2}{$z_2$}
    \psfrag{p1}{$p_1$}
    \psfrag{p2}{$p_2$}
     \psfrag{C'}{$C'$}
    \psfrag{C''}{$C''$}
    \includegraphics[scale=0.8]{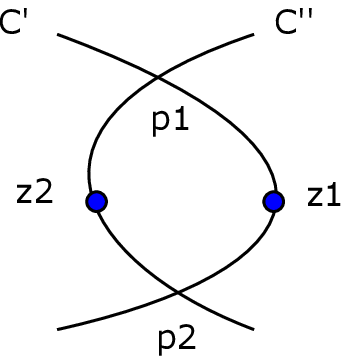}
    \caption{\label{fig:banana-ee} A curve in $\BHH(2,2)^{\odd}$}
    \end{figure}

In other words, 
$2z_1 \sim p'_1 + p'_2$ in $C'$ and $2z_2 \sim p''_1 + p''_2$ in $C''$ with the residue condition $\Res_{p'_i}(\wt{\omega}) +\Res_{p''_i}(\wt{\omega}) =0$ for $i=1, 2$. It follows from Theorem~\ref{thm:hodge} that 
$(C, \omega) \in \BHH(2,2)$. Note that $\HH(2,2)$ has two connected components $\HH(2,2)^{\hyp}$ and $\HH(2,2)^{\odd}$. In this example, $C$ is in the closure of locus of genus three hyperelliptic curves. However, by assumption $z_1$ and $z_2$ are ramification points of the corresponding 
admissible double cover, hence they are not conjugate under the hyperelliptic involution. We thus conclude that $(C, \omega) \in \BHH(2,2)^{\odd}$ and $(C, \omega) \not\in \BHH(2,2)^{\hyp}$. 
\end{example}

\begin{remark}
In Theorem~\ref{thm:hodge}, the zero orders $b'$ and $b''$ on both branches of a holomorphic node matter, not only their sum. 
For example, translation surfaces in $\HH(2)$ \emph{cannot} degenerate to two flat tori $E_1$ and $E_2$ attached at one point $q$ such that both have \emph{nonzero} area (in this case $b' = b'' = 0$, hence Theorem~\ref{thm:hodge} only implies smoothing into $\HH(1,1)$). This is because the dualizing line bundle restricted to $E_i$ is $\OO_{E_i}(q)$, which has degree one. In particular, it cannot have a double zero, unless the stable differential vanishes entirely on $E_i$. However, if we forget the flat structure and only keep track of the limit position of the double zero, using the notion of twisted canonical divisors (Section~\ref{subsec:twisted}) we will see that points in $E_i$ that are $2$-torsions to $q$ appear as all possible limits of the double zero.  
\end{remark}

We have discussed the case when $\omega$ has isolated zeros. If $\omega$ vanishes on a component of $C$, we can obtain a similar result
by tracking the zero orders on the branches of the nodes contained in the complement of the vanishing component. For ease of statement, let us deal with the case when 
$\omega$ vanishes on only one component $C'$, where $C'$ is a connected subcurve of $C$. The general case that $\omega$ vanishes on more components can be similarly tackled without further difficulty. 

Let $C'' = \overline{C\backslash C'}$ and $C'\cap C'' = \{ q_1, \ldots, q_m \}$. 
Since $\omega$ vanishes on $C'$, all $q_1, \ldots, q_m$ are holomorphic nodes. 
In the normalization $\wt{C}$, let $q'_l$ and $q_l''$ be the preimages of $q_l$ contained in $\wt{C}'$ and in $\wt{C}''$, respectively, for $l=1,\ldots, m$.  

Consider $\wt{\omega}$ restricted to $\wt{C}''$ such that 
$$ (\wt{\omega}|_{\wt{C}''})_0 - (\wt{\omega}|_{\wt{C}''})_{\infty} = \sum_i a_i z_i + \sum_j (b'_j h'_j + b''_j h''_j) - \sum_k (p'_k + p''_k) + \sum_l c''_l q_l'' $$
where $z_i$ are the isolated zeros of $\wt{\omega}$ in the smooth locus of $C$, $h'_j, h''_j$ are the preimages of the node $h_j$ that is not a pole of $\wt{\omega}$ and not contained in $C'$, and $p'_k, p''_{k}$ are the simple poles of $\wt{\omega}$ on the preimages of the node $p_k$. As before, 
$a_i \geq 1$ is the vanishing order of $\wt{\omega}$ at $z_i$, $b_j', b''_j \geq 0$ are the vanishing orders of $\wt{\omega}$ on the preimages $h_j', h_j''$ of the holomorphic node $h_j$, respectively, and $c''_l\geq 0$ is the vanishing order of $\wt{\omega}$ on the preimage $q''_l$ 
of $q_l$ contained in $\wt{C}''$. 

Suppose the arithmetic genus of $C'$ is $g'$ and let $\mu'$ be a partition of $2g' - 2$ such that 
$C'$ admits a differential $\omega'$ with signature $(c'_1, \ldots, c'_l, d_1, \ldots, d_s)$, where 
$c'_l$ is the vanishing order of $\omega'$ at $q'_l$ and $d_1, \ldots, d_s$ are the vanishing orders of $\omega'$ 
at the zeros other than the $q'_l$ in $C'$.   

\begin{corollary}
\label{cor:hodge-component}
In the above setting, we have 
$$(C, \omega) \in \BHH(\cdots, a_i, \cdots, b'_j+1, b''_{j}+1, \cdots, c'_l+1, c''_l+1, \cdots, d_r, \cdots). $$
\end{corollary}

\begin{proof}
Consider the nodal flat surface given by $\omega'$ on $C'$ and $\omega$ on $C''$ union at $q_1, \ldots, q_m$. 
Apply the local operations as in the proof of 
Theorem~\ref{thm:hodge} to smooth out $z_i, h_j, p_k, q_l$ and $d_t$ into the desired stratum. Meanwhile, scaling $\omega'$ 
by $t\cdot \omega'$ as $t\to 0$, the area of the flat surface $(C', t\cdot \omega')$ tends to zero while $\omega$ on $C''$ remains unchanged. The limit flat surface restricted to $C'$ corresponds to the identically zero differential on $C'$, hence equal to $\omega|_{C'}$. We 
thus obtain $(C, \omega)$ as a degeneration of differentials in the desired stratum. 
\end{proof}

\begin{example}
Let $C$ consist of two smooth curves $C'$ and $C''$, both of genus two, attached at a node $q$. Let 
$q'\in C'$ and $q''\in C''$ be the preimages of $q$ in the normalization of $C$. Let $\omega$ be a section of $K_C$, identified with 
a stable differential $\wt{\omega}$ on the normalization of $C$,  
such that 
$\omega|_{C'}\equiv 0$ and $(\wt{\omega}|_{C''})_0 = 2z''$ for a smooth point $z''\in C''$. In this case, $g' = 2$. Take 
$\mu' = (2)$ and $\omega' \in \HH(2)$ on $C'$ 
such that $(\omega')_0 = 2z'$ for a smooth point $z' \in C'$, see Figure~\ref{fig:2112}. 

 \begin{figure}[h]
    \centering
    \psfrag{z'}{$z'$}
    \psfrag{z''}{$z''$}
    \psfrag{q}{$q$}
     \psfrag{C'}{$C'$}
    \psfrag{C''}{$C''$}
    \includegraphics[scale=0.5]{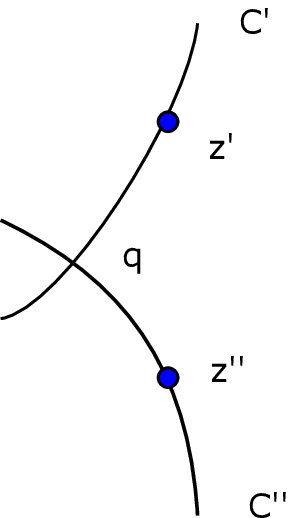}
    \caption{\label{fig:2112} A curve in $\BHH(2, 1, 1, 2)$}
    \end{figure}

 Then we have $c' = c'' = 0$ and $a = d = 2$ in the above notation. 
 By Corollary~\ref{cor:hodge-component}, we conclude that $(C, \omega) \in \BHH(2, 1, 1, 2)$. 
\end{example}

\section{Degenerations in the Deligne-Mumford space}
\label{sec:canonical}

As we have seen, a stable differential in the closure $\BHH(\mu)\subset \BHH$ may vanish entirely on a component of the underlying curve. In this case when Abelian differentials in $\HH(\mu)$ degenerate to it, we lose the information about the vanishing component as well as the limit positions of the zeros. Below we describe a refined compactification that resolves this issue. Note that modulo scaling, an Abelian differential is uniquely determined by its zeros, i.e. the corresponding canonical divisor. Recall that $\calP(\mu)$ parameterizes canonical divisors with signature $\mu$. Viewing it as a subset in $\MM_{g,n}$ by marking the zeros, we can take the closure $\BPP(\mu)\subset \BM_{g,n}$. The question reduces to analyzing which stable pointed curves appear in the boundary of $\BPP(\mu)$. 

Inspired by the theory of limit linear series, we introduce the notion of twisted canonical divisors, first on curves of pseudocompact type. The upshot is that when canonical divisors in $\calP(\mu)$ degenerate from underlying smooth curves to a curve of pseudocompact type, the limit object in $\BM_{g,n}$ must be a twisted canonical divisor. Conversely, in a number of cases twisted canonical divisors do appear as such limits, but not always. 
In the end we also discuss how to extend this notion to curves of non-pseudocompact type. 

\subsection{Twisted canonical divisors}
\label{subsec:twisted}

For the reader to get a feel, let us begin with curves of compact type with only one node. 
Suppose that a curve $C$ has a node $q$ connecting two smooth components $C_1$ and $C_2$ of genera $g_1$ and $g_2$, respectively, with 
$g_1 + g_2 = g$. Moreover, suppose that $C$ is the limit of a family of smooth genus $g$ curves $C_t$ over a punctured disk $T$, see Figure~\ref{fig:family}. 

 \begin{figure}[h]
    \centering
    \psfrag{T}{$T$}
    \psfrag{C2}{$C_2$}
    \psfrag{q}{$q$}
     \psfrag{Ct}{$C_t$}
    \psfrag{C1}{$C_1$}
    \includegraphics[scale=0.7]{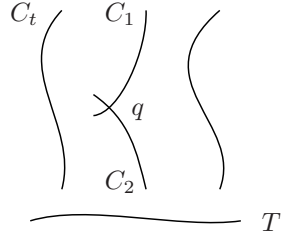}
    \caption{\label{fig:family} A curve of compact type in a family of curves}
    \end{figure}

Let $\calX\to T$ be the universal curve. The dualizing line bundle $K_{C}$ serves as a limit of canonical line bundles  
$K_{C_t}$ as $t\to 0$. However, this limit is \emph{not} unique. Becuase for any $m_i \in \bbZ$, 
\begin{eqnarray}
\label{eq:twisted-bundle}
\OO_{\calX}\left(\sum_{i=1}^2 m_i C_i\right)|_{C}\otimes K_C
\end{eqnarray}
is also a limit of $K_{C_t}$. 

Observe that 
$$\OO_{\calX}( C_1+C_2)|_C \cong \OO_C,$$ 
$$\OO_{\calX}(C_2)|_{C_1} \cong \OO_{C_1}(q),$$ 
$$\OO_{\calX}(C_1)|_{C_1} \cong \OO_{C_1}(-q).$$ 
Hence restricted to $C$, such limit line bundles are determined by the twisting coefficients $m_i$ and independent of 
the smoothing family $\calX$. We say that 
the line bundle in \eqref{eq:twisted-bundle}
is a \emph{twisted canonical line bundle} on $C$. Equivalently in this case, a twisted canonical line bundle consists of the data 
$$(K_{C_1}(a_1q), K_{C_2}(a_2q)),$$ 
where  
$a_1+a_2 = 2$. From the viewpoint of differentials, $K_{C_i}(a_iq)$ is the sheaf of meromorphic differentials on $C_i$ that are holomorphic away from $q$ and have pole order at most $a_i$ at $q$. 

The degree of $K_{C_i}(a_iq)$ is $d_i = 2g_i - 2 + a_i$. We say that $(K_{C_1}(a_1q), K_{C_2}(a_2q))$ has 
\emph{bidegree} $(d_1, d_2)$, where $d_1 + d_2 = 2g-2$. In particular, the dualizing line bundle $K_C$ corresponds to $(K_{C_1}(q), K_{C_2}(q))$ of bidegree $(2g_1-1, 2g_2-1)$. Note that knowing either one of the $a_i$ or one of the $d_i$ suffices to determine
a twisted canonical line bundle on $C$. 

\begin{remark}
\label{rem:twisted-pseudo}
Recall that a nodal curve is called of pseudocompact type, if each of its nodes is either separating or is an internal node of an irreducible component. Curves of compact type are special cases of pseudocompact type. As in the theory of limit linear series (\cite[p. 265-266]{HarrisMorrison}), the above analysis also applies to curves of pseudocompact type by treating $K_{C_i}$ as the dualizing line bundle of an irreducible component $C_i$, if $C_i$ contains self-nodes, and $g_i$ stands for the \emph{arithmetic} genus of $C_i$. In contrast, for a curve of non-pseudocompact type, even if we fix the line bundles restricted to each of its components, in general they do not determine the total line bundle. There are extra so called \emph{enrich structures} coming into play (see \cite{Maino}), which depend on first-order deformations of the curve. 
\end{remark}

Now suppose that a family of canonical divisors $m_1 z_1(t) + \cdots + m_n z_n(t)$ on $C_t$ degenerate to 
$m_1 z_1 + \cdots + m_n z_n$ on $C$, where all the $z_i$ are contained in the smooth locus of $C$. Define 
$$M_i = \sum_{z_j \in C_i} m_j$$ 
for $i=1,2$, measuring the total vanishing orders of the limit zeros in $C_i$. We use 
 ``$\sim$'' to denote \emph{linear equivalence} between two divisors on a curve. 

\begin{proposition}
\label{prop:limit-canonical}
In the above setting, we have 
\begin{eqnarray}
\label{eq:limit-canonical}
 \sum_{z_j\in C_i} m_j z_j + (2g_i - 2 - M_i)q \sim K_{C_i} 
 \end{eqnarray}
for $i=1, 2$. 
\end{proposition}

In particular, this proposition proves the ``only if'' part of Theorem~\ref{thm:canonical}. 

\begin{proof}
Let $\pi: \calX \to T$ be the universal curve and $\omega_{\calX/T}$ the relative dualizing line bundle. Denote by $Z_i$ the section 
of $\pi$ corresponding to the zero $z_i(t)$ for $i=1,\ldots,n$. 
Define a line bundle $\calL$ on $\calX$ by 
$$ \calL:= \omega_{\calX/T} \otimes \OO_{\calX}((M_1 - 2g_1+1)C_2))\otimes \OO_{\calX}\left(-\sum_{i=1}^n m_iZ_i\right). $$
Note that for $t\neq 0$, $\sum_{i=1}^n m_iz_i(t)$ is a canonical divisor 
of $C_t$, and $C_t$ and $C$ are disjoint. It follows that 
$$\calL|_{C_t} = K_{C_t}\left(-\sum_{i=1}^n m_iz_i(t)\right) = \OO_{C_t}. $$ 
Moreover,  
$$\calL|_{C_i} = K_{C_i} \left( (M_i - 2g_i+2) q - \sum_{z_j\in C_i} m_j z_j\right) $$
for $i=1,2$. 

Consider the direct image sheaf $\pi_{*}\calL$. For $t\neq 0$, 
$$(\pi_{*}\calL)|_{C_t} = H^0(C_t, \OO) = \bbC.$$ 
It follows from semicontinuity that 
$h^0(\calL|_{C}) \geq 1$. Note that $\calL|_{C_i}$ is a line bundle of degree zero for $i=1,2$. We claim that $\calL|_{C_i}$ is the trivial line bundle on $C_i$. Prove by contradiction. If say $\calL|_{C_1} \neq \OO_{C_1}$, then for any section $\sigma\in H^0(\calL|_{C})$, $\sigma|_{C_1} \in H^0(\calL|_{C_1})$, hence $\sigma|_{C_1} \equiv 0$. In particular, $\sigma(q) = 0$, hence $\sigma|_{C_2} \equiv 0$ and $\sigma \equiv 0$ on $C$, contradicting that $h^0(\calL|_{C}) \geq 1$. 

Since $\calL|_{C_i} = \OO_{C_i}$, it follows that 
$$ K_{C_i} \sim  (2g_i-2-M_i ) q + \sum_{z_j\in C_i} m_j z_j $$
for $i=1,2$, thus proving the proposition. 
\end{proof}

From the viewpoint of differentials, Relation \eqref{eq:limit-canonical} means there exists a (possibly meromorphic) differential on $C_i$ such that its zero order at $z_j\in C_i$ is $m_j$ and its zero or pole order at $q$ is $2g_i-2-M_i$. Note that 
$$(2g_1-2 - M_1) + (2g_2 -2 - M_1) = -2. $$
Relation \eqref{eq:limit-canonical} also implies that $2g_i-2-M_i \neq -1$, for otherwise $q$ is a base point of $K_{C_i}(q)$ and hence some zero 
$z_j$ would coincide with $q$, leading to a contradiction. As a consequence, exactly one of $2g_i-2-M_i$ is negative, and we call the corresponding component $C_i$ a \emph{polar component} of $C$. 

For $(C, z_1, \ldots, z_n)$ satisfying \eqref{eq:limit-canonical}, we say that $\sum_{i=1}^n m_i z_i$ is a \emph{twisted canonical divisor}. By definition, a  twisted canonical divisor uniquely determines the corresponding twisted canonical line bundle whose restriction to $C_i$ is the line bundle 
$$\OO_{C_i}\left(\sum_{z_j\in C_i} m_j z_j\right). $$
Note that $C_i$ is a polar component if and only if the restriction of the twisted canonical line bundle has degree $M_i$ strictly bigger than $2g_i - 2$. 

The essence of Proposition~\ref{prop:limit-canonical} says that limits of canonical divisors in $\BPP(\mu)$ are twisted canonical divisors in the locus of curves of pseudocompact type. Clearly the concepts of twisted canonical line bundles and twisted canonical divisors as well as Proposition~\ref{prop:limit-canonical} can be generalized without any difficulty to curves of pseudocompact type with arbitrarily many nodes. Nevertheless, for curves of pseudocompact type with more nodes, there exist twisted canonical divisors that do not appear as limits of ordinary canonical divisors. This is one place where the theory of limit linear series can help us extract more delicate information (see Example~\ref{ex:dimension-bound} and Proposition~\ref{prop:g=2}). 

\begin{remark}
If the universal curve $\calX$ is not smooth at a \emph{separating} node of $C$, by blowing up and making finite base change successively, we can resolve the singularity. The resulting special fiber amounts to inserting chains of rational curves between the two components connected by the node. In this case, the curve is still of pseudocompact type, hence the above argument works and we can deduce the same result (see \cite[Theorem 2.6 and Remark after]{EisenbudHarrisLimit}). 
In addition, we do not have to resolve the surface singularity at a self-node, because locally at such a node the corresponding component of $C$ 
is still a \emph{Cartier} divisor of $\calX$ (see e.g. \cite[Section 1.3]{EstevesMedeiros}). However, for a curve of non-pseudocompact type, inserting chains of rational curves at a non-separating and external node may change significantly the possible types of twisted canonical divisors. We analyze this issue in detail in Section~\ref{subsec:non-compact}. 
\end{remark}

\subsection{Dimension bounds on spaces of twisted canonical divisors}
\label{subsec:dimension}

Recall that $\mu = (m_1, \ldots, m_n)$ is a partition of $2g-2$. Let $\pi: \calX\to B$ be a smoothing family of genus $g$ curves of compact type with $n$ sections $z_1, \ldots, z_n$, in the sense of \cite[p. 354]{EisenbudHarrisLimit}.
Inspired by \cite[Theorem 3.3]{EisenbudHarrisLimit}, we show that there exists a variety $\calP(\calX/B; \mu)$ parameterizing $(C, z_1, \ldots, z_n)\in B$ such that $\sum_{i=1}^n m_i z_i$ is a twisted canonical divisor on $C$. 
Moreover, $\calP(\calX/B; \mu)$ has a \emph{determinantal structure}, which gives rise to a lower bound for every irreducible component of $\calP(\calX/B; \mu)$. 


\begin{theorem}
\label{thm:scheme-twisted}
There exists a variety $\calP(\calX/B; \mu)$ over $B$, compatible with base change, whose point over any $q$ in $B$ (if not empty) corresponds to a twisted canonical divisor given by $m_1 z_1(q) + \cdots m_n z_n (q)$. Furthermore, every irreducible component of $\calP(\calX/B; \mu)$ has dimension $\geq \dim B - g$. 
\end{theorem}

\begin{proof}
Our argument follows from \cite[Proof of Theorem 3.3]{EisenbudHarrisLimit}. Let $X_q$ be the fiber of $\calX$ over $q$. 
If in $\calX$ no nodes are smoothed, it is clear how to define the variety of twisted canonical divisors, by taking the union of strata 
of ordinary (possibly meromorphic) canonical divisors on each component of $X$. 

Suppose now some of the nodes of fibers of $\pi$ are smoothed in the generic fiber. It suffices to deal with the case when 
the general fiber of $\calX$ is smooth. Then for an arbitrary smoothing family, one regards the family as being obtained from 
several families, in each of which all nodes are smoothed (see \cite[Figure in p. 355]{EisenbudHarrisLimit}). 

It remains to prove the result when the general fiber of $\calX$ is smooth. Since the problem is local on $B$, we assume that $B$ is affine. Take a relatively ample divisor $D$ on $\calX$ such that $D$ is contained in the smooth locus of $\pi$ and disjoint from the sections $z_i$. Replacing $D$ with a high multiple of itself, we may assume that it intersects every component of a reducible fiber with high degree. Denote by $d$ the total degree of $D$ relative over $B$. 

Let $\omega_{\mu}$ be a twisted relative canonical line bundle on $X$ such that restricted to each fiber 
$(C, z_1, \ldots, z_n)$ it is the unique twisted canonical line bundle of multi-degree $\sum_{z_j\in C_i} m_j$ on every component 
$C_i$ of $C$. The existence of $\omega_{\mu}$ is explained in \cite[p. 359]{EisenbudHarrisLimit}. It follows that 
$\pi_{*} \omega_{\mu}(D)$ is a vector bundle of rank 
$$1-g+ (2g-2 + d) = g-1 +d$$ 
by Riemann-Roch. Let $\calP'$ be the corresponding projective bundle with fiber dimension $g-2 + d$ over $B$. A point of $\calP'$ over $q' \in B$ is thus a section 
$\sigma \in H^0( \omega_{\mu}(D)|_{X_{q'}})$, up to the equivalence $\sigma \sim \lambda \sigma$ for a nonzero scalar $\lambda$. 

Consider the subvariety $\calP'(\calX/B; \mu)$ in $\calP'$ cut out by the following groups of equations: 
\begin{itemize}
\item \emph{Vanishing on $D$}.  We require that $\sigma$ vanishes on $D$. 
\item \emph{Ramification at $z_i$}. For each section $z_i$, $\sigma$ vanishes on $z_i$ with multiplicity $\geq m_i$. 
\end{itemize}
 
The vanishing condition on $D$ is given by $d$ equations. The ramification condition at each $z_i$ imposes $m_i$ equations, so the total ramification conditions impose $\sum_{i=1}^n m_i = 2g-2$ equations. It follows that the dimension of every irreducible component of $\calP'(\calX/B; \mu)$ is at least 
$$ \dim B + (g-2 + d) - d - (2g-2) = \dim B - g. $$
Let $U$ be an irreducible component of $\calP'(\calX/B; \mu)$. For a general point $\sigma \in U$, there are two possibilities. First, if $\sigma$ vanishes on a component of the underlying curve $C$, then this property holds for all sections parameterized in $U$. The other case is that $\sigma$ has isolated zeros in $C$. Then it follows from the construction that 
$$(\sigma)_0 = \sum_{i=1}^n m_i z_i + D,$$ 
which gives rise to a twisted canonical divisor by subtracting the fixed part $D$. Conversely, a twisted canonical divisor determines such a $\sigma$ up to scaling, where $\sigma$ has isolated zeros in $C$. By collecting irreducible components of $\calP'(\calX/B; \mu)$ of the latter type, we thus obtain the desired $\calP(\calX/B; \mu)$ parameterizing twisted canonical divisors with signature $\mu$. 
\end{proof}

We present an example to show that the dimension bound in Theorem~\ref{thm:scheme-twisted} can be attained. 

\begin{example}
\label{ex:dimension-bound}
Consider $\mu = (1,1)$ in $g=2$. Let $B_1$ be the locus of curves of compact type in $\BM_{2,2}$. Let $U_1$ be the irreducible locus in $\BM_{2,2}$ parameterizing pointed curves $C$ that consist of two elliptic curves $E_1$ and $E_2$, connected by a rational curve $R$, where the two marked points $z_1$ and $z_2$ are contained in $R$. Let $q_i = E_i \cap R$ for $i=1,2$, see Figure~\ref{fig:101}.  
 \begin{figure}[h]
    \centering
     \psfrag{z1}{$z_1$}
    \psfrag{z2}{$z_2$}
   \psfrag{q1}{$q_1$}
    \psfrag{q2}{$q_2$}
     \psfrag{E1}{$E_1$}
       \psfrag{R}{$R$}
    \psfrag{E2}{$E_2$}
    \includegraphics[scale=1.0]{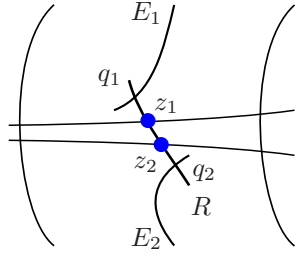}
       \caption{\label{fig:101} Two elliptic curves connected by a rational curve with two simple zeros }
\end{figure} 

Since $z_1 + z_2 \sim q_1 + q_2$ in $R\cong \bbP^1$, elements in $U_1$ are twisted canonical divisors with signature $(1,1)$. Note that $\dim B_1 = \dim \BM_{2,2} = 5$ and $\dim U_1 = 3$. 
Hence the dimension of $U_1$ is equal to $\dim B_1 - g$. 

Similarly for $\mu = (2)$ in $g=2$, consider $U_2 \subset \BM_{2,1}$ parameterizing the same underlying curve as in $U_1$, 
where the unique marked point $z$ is contained in $R$, see Figure~\ref{fig:1012}. 
 \begin{figure}[H]
    \centering
     \psfrag{z}{$z$}
      \psfrag{R}{$R$}
   \psfrag{q1}{$q_1$}
    \psfrag{q2}{$q_2$}
     \psfrag{E1}{$E_1$}
    \psfrag{E2}{$E_2$}
    \includegraphics[scale=1.0]{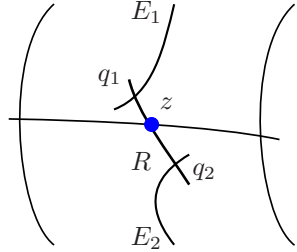}
       \caption{\label{fig:1012} Two elliptic curves connected by a rational curve with a double zero }
\end{figure}  
Since $2z \sim q_1 + q_2$ in $R$, elements in $U_2$ are twisted canonical divisors with signature $(2)$. Note that $\dim B_2 = \dim \BM_{2,1} = 4$ and $\dim U_2 = 2$. Hence the dimension of $U_2$ is equal to $\dim B_2 - g$. 
\end{example}

We further classify which twisted canonical divisors in the above examples come from degenerations of ordinary canonical divisors. 

\begin{proposition}
\label{prop:g=2}
In the above setting, $(C, z_1, z_2)\in U_1$ is contained in $\BPP(1,1)$ if and only if $z_1+z_2$ is a section in the linear series $g^1_2$ on $R$ induced by $2q_1 \sim 2q_2$. On the other hand, $U_2$ is disjoint with $\BPP(2)$. 
\end{proposition}

\begin{proof}
First consider $(C, z_1, z_2)$ in $U_1$. If it is a degeneration of canonical divisors from $\calP(1,1)$, the canonical limit series on $C$ possesses a section $z_1+z_2$ in its aspect $g^1_2$ on $R$. By the compatibility condition on vanishing sequences (see Section~\ref{subsec:lls}), the aspect on $E_i$ has vanishing sequence 
$(0, 2)$ at $q_i$ for $i=1, 2$. Hence the vanishing sequences of the aspect on $R$ at $q_1$ and $q_2$ are both equal to $(0, 2)$. It implies that 
the aspect $g^1_2$ on $R$ is induced by $2q_1\sim 2q_2$, i.e. a double cover of $\PP^1$ ramified at $q_1, q_2$ and mapping $z_1, z_2$ to the same image. Conversely, if $z_1 + z_2$ is a section of such $g^1_2$ on $R$, using either the smoothability result of limit linear series or admissible double covers, we see that $(C, z_1, z_2)$ can be smoothed into $\calP(1,1)$. 

The same argument works for $(C, z)$ in $U_2$. If it was contained in $\calP(2)$, the limit $g^1_2$ on $R$ would be a double cover of $\PP^1$ ramified at $q_1, q_2$ and $z$, contradicting the fact that such a cover has only two ramification points.  
\end{proof}

The proposition implies that there exist twisted canonical divisors that \emph{cannot} be smoothed. Before studying in general which twisted canonical divisors appear in the boundary of $\BPP(\mu)$, we point out that as a component of space of twisted canonical divisors over $\MM_{g,n}$, the dimension of $\calP(\mu)$ is indeed \emph{one larger} than the dimension bound in Theorem~\ref{thm:scheme-twisted}: 
$$ \dim \calP(\mu) = 2g + n -2 = \dim \MM_{g,n} - (g-1). $$
The reason is because on a genus $g$ curve, the dimension of space of holomorphic sections for the canonical line bundle is \emph{one larger} than the other line bundles of degree $2g-2$. Under certain extra assumptions, we can take this into account and come up with a refined dimension bound as follows. 

\begin{proposition}
\label{prop:dimension-refined}
In the setting of Theorem~\ref{thm:scheme-twisted}, suppose that the twisted canonical line bundle on each fiber curve over $B$ has 
exactly $g$-dimensional space of sections. Then the dimension of every irreducible component of $\calP(\calX/B; \mu)$ is at least 
$$\dim B - (g-1). $$
\end{proposition}

Compared to Theorem~\ref{thm:scheme-twisted}, we say that $\dim B - (g-1)$ is the \emph{refined} dimension bound. We remark that in the case $B = \MM_{g,n}$, i.e. smooth curves with $n$ marked points, $\calP(\calX/B; \mu)$ is just the stratum $\calP(\mu)$, hence has dimension equal to the refined dimension bound. 

\begin{proof}
The argument is similar to the proof of Theorem~\ref{thm:scheme-twisted}. The only difference is that we do not twist by a very ample divisor $D$. Instead, 
by assumption $\pi_{*}\omega_{\mu}$ is a vector bundle of rank $g$ over $B$. After projectivization and imposing the vanishing conditions on $z_i$, we obtain the desired dimension lower bound as
$$ \dim B + (g-1) - (2g-2) = \dim B - (g-1). $$
\end{proof}

\begin{corollary}
\label{cor:dimension-one}
In the setting of Theorem~\ref{thm:scheme-twisted}, suppose that $B$ parameterizes curves of compact type that have at most one node $q$ such that $h^0(C_i, (2g_i-2-M_i)q) = 1$ for any holomorphic component $C_i$. Then every irreducible component of $\calP(\calX/B; \mu)$ has dimension at least 
$$\dim B - (g-1).$$ 
\end{corollary}

\begin{proof}
Suppose $C = C_1\cup_q C_2$ is a nodal curve in $B$. Recall that $M_i$ is the sum of zero orders for those zeros contained in $C_i$. 
Without loss of generality, assume that $M_1 > 2g_1-2$ and $M_2 \leq 2g_2 -2 $. Then $C_1$ is a polar component and $C_2$ is holomorphic. Let $K_i$ be the restriction of the corresponding twisted canonical line bundle $K_{\mu}$ to $C_i$, i.e. 
$$ K_i = K_{C_i}( (M_i - 2g_i+2)q), $$
which has degree $M_i$ for $i=1,2$. 
Since $M_1 > 2g_1 - 2$, by Riemann-Roch we have 
$$ h^0(C_1, K_1) = 1 - g_1 + M_1. $$
On $C_2$, we have 
$$ h^0(C_2, K_2) = h^0(C_2, (2g_2-2-M_2)q) + 1 - g_2 + M_2 = 2-g_2 +M_2. $$
It implies that 
$$ h^0(C, K_{\mu}) = h^0(C_1, K_1) + h^0(C_2, K_2) - 1 = g, $$
where we subtract one because gluing two sections of $K_1$ and $K_2$ at $q$ imposes one condition. 
Now the desired dimension bound follows from Proposition~\ref{prop:dimension-refined}. 
\end{proof}

\subsection{Smoothing twisted canonical divisors}
\label{subsec:smoothing}

Inspired by \cite[Theorem 3.4]{EisenbudHarrisLimit}, using dimension bounds on $\calP(\calX/B;\mu)$ 
we obtain the following smoothability result of twisted canonical divisors. 

\begin{proposition}
\label{prop:smoothing}
In the setting of Proposition~\ref{prop:dimension-refined}, suppose $(C, z_1, \ldots, z_n)\in B$ is contained in an irreducible component $U$ of $\calP(\calX/B; \mu)$ such that $\dim U = \dim B - (g-1)$, i.e. the refined dimension bound is attained. Then $(C, z_1, \ldots, z_n)\in \BPP(\mu)$. 
\end{proposition}

\begin{proof}
We adapt the proof of \cite[Theorem 3.4]{EisenbudHarrisLimit} to our case. 
Let $\wt{\calX} \to \wt{B}$ be the versal family of pointed curves around $(C, z_1, \ldots, z_n)$. Let $f: B \to \wt{B}$ be the map 
locally inducing $\calX \to B$ with $n$ sections of marked points. Let $\wt{U}$ be a component of $\calP(\wt{\calX}/ \wt{B}; \mu)$ such that 
$U$ is a component of $f^{*}\wt{U}$ and let $\wt{C}$ be the point in $\wt{U}$ corresponding to the 
pointed curve $C$. Since the dimension of space of sections does not drop under specialization, the family $\wt{\calX} \to \wt{B}$ satisfies 
the assumption in Proposition~\ref{prop:dimension-refined}, hence by the refined dimension bound we have 
$$\dim \wt{U} \geq \dim \wt{B} - (g-1).$$ 

If $\wt{U}$ does not lie entirely in the discriminant locus of $\wt{\calX}/\wt{B}$ parameterizing nodal curves, then it implies that a general point 
of $\wt{U}$ parameterizes an ordinary canonical divisor with signature $\mu$ on a smooth curve, and hence we are done. Suppose on the contrary 
that $\wt{U}$ lies over a component $\wt{B}'$ of the discriminant locus. Note that $\wt{B}'$ 
is a hypersurface in $\wt{B}$, hence 
$$\dim \wt{U} \geq \dim \wt{B} - (g-1) > \dim \wt{B}' - (g-1).$$
Thus every component of $f^{*}\wt{U}$, including $U$, has dimension $> \dim B - (g-1)$, contradicting the assumption 
that  $\dim U = \dim B - (g-1)$. 
\end{proof}

\begin{example}
Consider the boundary divisor $B = \Delta_{1;\{ 1\}}$ in $\BM_{g,1}$ parameterizing curves $C = C_1\cup_q C_2$, where $C_1$ has genus $g_1 = 1$, $C_2$ has genus $g_2 = g-1$, 
and the unique marked point $z$ is contained in $C_2$, see Figure~\ref{fig:1g-11}. 
 \begin{figure}[h]
    \centering
     \psfrag{C1}{$C_1$}
      \psfrag{C2}{$C_2$}
   \psfrag{q}{$q$}
    \psfrag{z}{$z$}
      \includegraphics[scale=0.7]{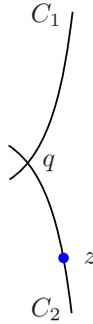}
       \caption{\label{fig:1g-11} An elliptic curve union a pointed curve of genus $g-1$ }
\end{figure}  
In this case $C_1$ is a holomorphic component with $M_1 = 0$ and $2g_1 - 2 - M_1 = 0$. 
Since $h^0(C_1,  \OO) = 1$, $B$ satisfies the assumption of Corollary~\ref{cor:dimension-one}. A twisted canonical divisor with signature 
$(2g-2)$ in this case is given by a curve in $\calP(2g-2, -2)$ glued to $C_1$ at the double pole. Hence we conclude that 
$$\dim \calP(\calX/B; 2g-2) = \dim \calP(2g-2, -2) + 1 = 2g - 2 = \dim B - (g-1). $$
By Proposition~\ref{prop:smoothing}, $\calP(\calX/B; 2g-2)$ is contained in $\calP(2g-2)$. Note that it proves the ``if'' part of 
Theorem~\ref{thm:canonical} in this special case. 

On the other hand, if $C_1$ contains $z$, then $M_2 = 0$, $C_2$ is a holomorphic component, and $2g_2 - 2 - M_2 = 2g-4$. But 
$h^0(C_2, (2g-4)q) > 1$ for any $g > 4$, hence Proposition~\ref{prop:smoothing} does not directly apply. 
\end{example}

Now let us prove the ``if'' part of Theorem~\ref{thm:canonical} in general. The upshot is a direct dimension count using the dimension of strata of meromorphic differentials, combined with an explicit deformation using the flat-geometric description of a higher order pole. 

\begin{proof}[Proof of Theorem~\ref{thm:canonical}, the ``if'' part]
Without loss of generality, assume that $C = (C_1\cup_q C_2, z_1, \ldots, z_n)$ such that 
$z_1, \ldots, z_k\in C_1$ and $z_{k+1}, \ldots, z_n \in C_2$, satisfying 
Relation~\eqref{eq:limit-canonical}, which is equivalent to 
$$(C_1, z_1, \ldots, z_k, q) \in \calP_1 := \calP(m_1, \ldots, m_k, 2g_1 - 2 - M_1), $$
$$(C_2, z_{k+1}, \ldots, z_n, q) \in \calP_2 := \calP(m_{k+1}, \ldots, m_n, 2g_2 - 2 - M_2). $$
Recall that $M_1 + M_2 = 2g-2$ and $g_1 + g_2 = g$. In particular, $2g_i - 2 - M_i = -1$ cannot hold 
for $i=1,2$, for otherwise $q$ would coincide with some $z_i$, because $K_C|_{C_i} = K_{C_i}(q)$ has a base point at $q$. 
It follows that exactly one of $\calP_1$ and $\calP_2$ is a stratum of meromorphic differentials with a pole at $q$ (see 
Section~\ref{subsec:pole}). Without loss of generality, further assume that $2g_1 - 2 - M_1 \geq 0$ and $2g_2 - 2 - M_2 < 0$. 
By \cite[Lemma 3.6]{Boissy}, we have   
\begin{eqnarray}
\label{eq:dimension-product}
\dim \calP_1 + \dim \calP_2  & = & (2g_1 - 2 + (k+1)) + (2g_2 - 3 + (n-k+1)) \\ \nonumber
 & = & 2g - 3 + n. 
\end{eqnarray}
Note that $\calP_i$ may have more than one connected component, but all components have the same dimension, hence all irreducible components of $\calP_1 \times \calP_2$ have dimension $2g-3+n$ given by \eqref{eq:dimension-product}. 

Let $\Delta_{g_1; \{1,\ldots, k\}}^{\circ}$ be the open dense subset of the boundary divisor $\Delta_{g_1; \{1,\ldots, k\}}$ that parameterizes curves with only one node, i.e. nodal curves of the same topological type as that of $C$. We have proved that 
$$\BPP(\mu) \cap \Delta_{g_1; \{1,\ldots, k\}}^{\circ} \subset \calP_1 \times \calP_2.$$
Suppose $U = U_1 \times U_2$ is an irreducible component of $\calP_1 \times \calP_2$ that intersects $\BPP(\mu)$ non-empty, where $U_i$ is an irreducible component of $\calP_i$. 
 Then 
\begin{eqnarray*}
\dim (U\cap \BPP(\mu)) & \geq & \dim \BPP(\mu) + \dim \Delta_{g_1; \{1,\ldots, k\}} - \dim \BM_{g,n}\\ 
  & = & \dim \BPP(\mu) - 1 \\ 
  & = & \dim U, 
\end{eqnarray*}
which implies that $U \subset \BPP(\mu)$ by the irreducibility of $U$. 

It remains to show that every irreducible component $U_1\times U_2$ of $\calP_1 \times \calP_2$ intersects $\BPP(\mu)$. We prove it by an explicit procedure, inspired by \cite{Boissy}. Take a translation surface $X = (C_1, z_1, \ldots, z_k, q)$ parameterized in $U_1$, where $q$ is the zero of order $2g_1 - 2 - M_1$. The flat-geometric local neighborhood of $q$ can be obtained by gluing $2g_1 - 1 -M_1$ disks $D_i$, each with a radius slit whose edges are labeled by $A_i$ and $B_i$, where $B_i$ is identified with $A_{i+1}$. Then the centers of the disks are identified as the same point $q$, with desired cone angle $(2g_1 - 1 -M_1)\cdot (2\pi)$.  
Now, expand the boundary of each $D_i$ to infinity such that $D_i$ becomes the standard Euclidean plane, while preserving the gluing pattern of the radius slits. At the same time, inside each $D_i$ remove a small neighborhood $E_i$ around $q$, where the $E_i$'s have suitable polygon boundaries and gluing pattern, such that the newborn $(2g_1 - 1 -M_1)$ broken Euclidean planes are glued to form a (possibly special) element $Y = (C_2, z_{k+1}, \ldots, z_n, q)$ in $U_2$, where the pole sits at the infinity of the $D_i$'s (identified altogether) with pole order $(2g_1 - 1 - M_1) + 1 = - (2g_2 - 2 - M_2)$ by \cite[Section 3]{Boissy}, see Figure~\ref{fig:expand-shrink}. 

 \begin{figure}[h]
    \centering
   \psfrag{q}{$q$}
    \psfrag{Ai}{$A_i$}
    \psfrag{Bi}{$B_i$}
     \psfrag{Di}{$D_i$}
      \psfrag{Ei}{$E_i$}
    \includegraphics[scale=0.7]{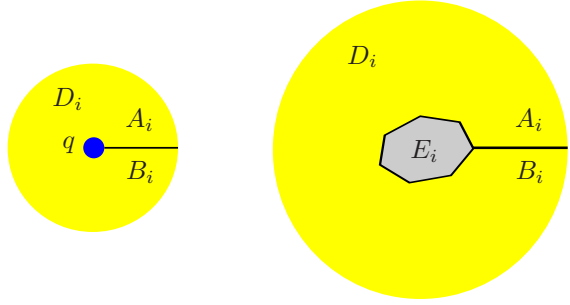}
       \caption{\label{fig:expand-shrink} Remove a neighborhood of a zero and expand to a pole}
\end{figure}

Note that in the intermediate stage when $D_i$ is relatively large compared to $E_i$, but still of finite area, the above operation provides a translation surface in $\HH(\mu)$, with zeros at $z_1, \ldots, z_n$ and no pole. If we shrink each $E_i$ to a point, it gives rise to $X \in U_1$. Alternatively, shrinking $E_i$ is equivalent to expanding $D_i$ modulo scaling, hence it also gives rise to $Y\in U_2$. Therefore, the limit object of this shrinking/expanding process 
consists of $X$ union $Y$ as a stable curve. In other words, we have exhibited  
an explicit family of translation surfaces in $\HH(\mu)$ degenerating to an element in $U_1\times U_2$, thus finishing the proof.
\end{proof}

\begin{example}
Consider a small flat torus $Y$ embedded in a big flat torus $X$ as in Figure~\ref{fig:embedded-tori}. 
 \begin{figure}[h]
    \centering
   \psfrag{X}{$X$}
    \psfrag{Y}{$Y$}
    \psfrag{z}{$z$}
    \includegraphics[scale=0.6]{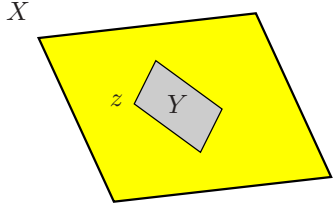}
       \caption{\label{fig:embedded-tori} A small flat torus embedded in a big one}
\end{figure}

The region $X\backslash Y$ represents a translation surface in $\HH(2)$, where the inner vertex $z$ is the double zero. 
Expanding the boundary of $X$ to infinity, we obtain a flat torus $E_1$ in $\HH(2, -2)$, where the double pole is at the infinity (see \cite[Figure 4]{Boissy}). Up to rescaling, this procedure is the same as shrinking $Y$ to a point $q$. We thus obtain another flat torus $E_2$ in $\HH(0)$, with $q$ is marked as the attachment point to  
the double pole of $E_1$. It matches with our theory of twisted canonical divisors applied to this case. 
\end{example}

\begin{example}
Let $C$ consist of two smooth curves $C_1$ and $C_2$, of genus one and two respectively, meeting at a node $q$. Take a Weierstrass point $z_2 \in C_2$ and a $2$-torsion point $z_1 \in C_1$ with respect to $q$, see Figure~\ref{fig:12}. 
 \begin{figure}[h]
    \centering
   \psfrag{C1}{$C_1$}
    \psfrag{C2}{$C_2$}
    \psfrag{q}{$q$}
    \psfrag{z1}{$z_1$}
    \psfrag{z2}{$z_2$}
    \includegraphics[scale=0.7]{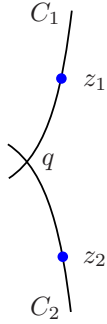}
       \caption{\label{fig:12} An elliptic curve marked at a $2$-torsion union a genus two curve marked at a Weierstrass point}
\end{figure}
In this setting, $2z_1 - 2q \sim K_{C_1}$ and $2z_2  \sim K_{C_2}$, hence 
Theorem~\ref{thm:canonical} implies that $(C, z_1, z_2)\in \BPP(2,2)$. 
\end{example}

Now let us consider curves of compact type with more nodes. Suppose a curve $C$ of compact type has $m+1$ irreducible components 
$C_0, \ldots, C_m$ of genus $g_0, \ldots, g_m$, respectively. Then $C$ possesses $m$ nodes
$q_1, \ldots, q_{m}$. The locus $B$ of such pointed curves in $\BM_{g,n}$ has dimension 
$$\dim B = 3g-3 + n - m. $$ 
As before, define $M_i = \sum_{z_j \in C_i} m_j$ for $i=0, \ldots, m$. We have 
$$\sum_{i=0}^m g_i = g, \quad \sum_{i=0}^g M_i = 2g-2. $$ 

Let $\calX$ be the universal $n$-pointed curve over $B$.
There exists a \emph{unique} twisted canonical line bundle $\calL$, 
 independent of $\calX$, 
such that the restriction $\calL|_{C_i}$ has degree $M_i$ for all $i$. 
Suppose that 
$$\calL|_{C_i} \otimes \OO_{C_i}\left(\sum_{q_j\in C_i} s_{ij} q_j \right) = K_{C_i}  $$ 
for some $s_{ij}\in \bbZ$. In particular, 
$$\sum_{q_j\in C_i} s_{ij} = 2g_i-2 - M_i.$$ 
Recall that $C_i$ is called a polar component of $C$, if at least one $s_{ij} < 0$. Conversely If all 
$s_{ij} \geq 0$, we call $C_i$ a holomorphic component. In other words, a component is holomorphic (resp. polar) 
if the corresponding twisted canonical divisor restricted to it is effective (resp. not effective). 

Suppose that $C$ possesses $N$ polar components. Let $q_j$ be a node connecting two components $C_i$ and $C_k$. 
Then $s_{ij} + s_{kj} = -2$. In particular, if $C_i$ is a \emph{tail} of $C$, i.e. if $C_i$ meets the closure of its complement in $C$ at only one point, 
then $s_{ij} \neq -1$, again because $q_j$ is a base point of $K_{C_i}(q_j)$. 
It immediately follows that 
\begin{eqnarray*}
\label{eq:negative}
N\leq m. 
\end{eqnarray*}
In other words, at least one component of $C$ is holomorphic. 

Consider the extremal case when $N = m$. 

\begin{theorem}
\label{thm:canonical-more}
For $(C, z_1, \ldots, z_n) \in \BM_{g,n}$ in the above setting, if it is contained in $\BPP(\mu)$, then 
\begin{eqnarray}
\label{eq:linear-more}
 \sum_{z_j\in C_i} m_j z_j +  \sum_{q_j\in C_i} s_{ij} q_j  \sim K_{C_i} 
 \end{eqnarray}
holds for all $i$. Conversely, if Relation~\eqref{eq:linear-more} holds and further assume that $C$ has only one holomorphic component, 
then it is contained in $\BPP(\mu)$. 
\end{theorem}

\begin{proof}
The first part of the result follows from the same proof as in 
Proposition~\ref{prop:limit-canonical}. For the other part, by assumption there exists at least one tail of $C$, denoted by $X$, which is a polar component. Let $Y = \overline{C\backslash X}$ and $q = X\cap Y$. Suppose $s_X < 0$ and $s_Y\geq 0$ are the signatures of the twisted canonical line bundle 
restricted to $X$ and to $Y$ at $q$. Suppose the genera of $X$ and $Y$ are $g_X$ and $g_Y$, respectively, and the vanishing orders of
the marked points in $X$ and in $Y$ are of signature $\mu_X$ and $\mu_Y$, respectively. We have $g = g_X + g_Y$, 
$\mu = (\mu_X, \mu_Y)$, and $s_X + s_Y = -2$. By Theorem~\ref{thm:canonical}, every one-nodal curve in $\calP(\mu_X, s_{X}) \times \calP(\mu_Y, s_Y)$ can be smoothed into 
$\calP(\mu)$. Now treat $q$ as a marked point in $Y$ with signature $s_Y\geq 0$. By assumption, all components of $Y$ but one are polar, hence the desired result follows from induction on the number of components of the underlying curve. 
\end{proof}

\begin{remark}
\label{rem:hodge-pseudo}
Theorem~\ref{thm:canonical-more} also holds for $C$ of pseudocompact type under the same assumption, by treating $K_{C_i}$ as the dualizing line bundle of $C_i$ and taking the $q_j$'s from the set of separating nodes of $C$ only. The same proof works for the first part as explained in Remark~\ref{rem:twisted-pseudo}. To see the smoothability part, note that at the two branches of a self-node the (stable) differential has simple poles. Therefore, one can first smooth out all self-nodes by the plumbing operation as in the proof of Theorem~\ref{thm:hodge}, and then reduce it to the case of compact type.  
\end{remark}

\begin{example}
Suppose a smooth curve $C_0$ is attached to $m$ tails $C_1, \ldots, C_m$ as in Figure~\ref{fig:backbone}. 
 \begin{figure}[h]
    \centering
   \psfrag{C0}{$C_0$}
    \psfrag{C1}{$C_1$}
    \psfrag{C2}{$C_2$}
     \psfrag{Cm}{$C_m$}
    \includegraphics[scale=0.8]{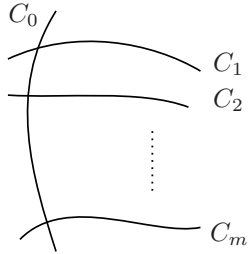}
       \caption{\label{fig:backbone} A curve with $m$ tails}
\end{figure}

Suppose further that the zero orders of marked points in 
$C_i$ add up to $d_i \geq 2g_i$ for $i>0$ and that they satisfy Relation~\eqref{eq:linear-more}. 
Then $C_i$ is a polar component for $i=1, \ldots, m$. Therefore, by 
Theorem~\ref{thm:canonical-more} this pointed curve is contained in $\BPP(\mu)$. 
\end{example}

\begin{example}
Let $C$ be a chain of elliptic curves $E_1, \ldots, E_g$ such that $q_i = E_i \cap E_{i+1}$ for $i=1,\ldots, g-1$. Suppose that $E_g$ contains a marked point $z$, see Figure~\ref{fig:elliptic-chain}.  
 \begin{figure}[h]
    \centering
   \psfrag{E1}{$E_1$}
    \psfrag{E2}{$E_2$}
    \psfrag{Eg-1}{$E_{g-1}$}
     \psfrag{Eg}{$E_g$}
      \psfrag{z}{$z$}
    \psfrag{q1}{$q_1$}
     \psfrag{qg-1}{$q_{g-1}$}
    \includegraphics[scale=1.0]{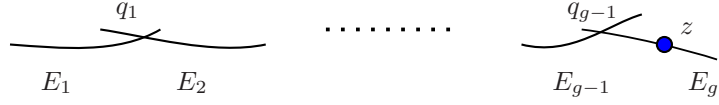}
       \caption{\label{fig:elliptic-chain} A chain of elliptic curves} 
 \end{figure}      

Further suppose that $(2g-2)z \sim (2g-2)q_{g-1}$ in $E_g$, $(2g-4)q_{g-1} \sim (2g-4)q_{g-2}$ in $E_{g-1}, \ldots, 2q_{2} \sim 2q_{1}$ in $E_2$. In this setting, $E_1$ is the only holomorphic component, and Relation~\eqref{eq:linear-more} holds on every component of $C$. By Theorem~\ref{thm:canonical-more}, we conclude that $(C, z) \in \BPP(2g-2)$. 
\end{example}

We remark that for a pointed curve of pseudocompact type with more than one separating node, Relation~\eqref{eq:linear-more} may fail to be sufficient for the containment in $\BPP(\mu)$, if the curve has more than one holomorphic component (see Example~\ref{ex:dimension-bound}). 

\subsection{Twisted meromorphic canonical divisors}
\label{subsec:meromorphic}

Take a sequence of integers $\mu = (k_1, \ldots, k_r, - l_1, \ldots, -l_s)$ such that $k_i \geq 0$, $l_j > 0$, 
$\sum_{i=1}^r k_i - \sum_{j=1}^s l_j = 2g-2$ and $s > 0$, i.e. we consider meromorphic differentials with at least one pole. We use $z_i$ to denote a zero of order $k_i$, and $p_j$ for a pole of order $l_j$. For pointed nodal curves of pseudocompact type, the notion of twisted canonical divisors can be defined in the same way in the meromorphic setting, by allowing sections of marked points to have negative coefficients if they correspond to poles. Below we establish similar dimension bounds for components of space of twisted meromorphic canonical divisors as well as smoothability results.   

\begin{theorem}
\label{thm:scheme-twisted-meromorphic}
Let $\pi: \calX\to B$ be a smoothing family of genus $g$ curves of compact type with $n=r+s$ sections $z_1, \ldots, z_r, p_1, \ldots, p_s$. 
Then there exists a variety $\calP(\calX/B; \mu)$ over $B$, compatible with base change, whose point over any $q$ in $B$ (if not empty) corresponds to a twisted meromorphic canonical divisor given by $k_1 z_1(q) + \cdots + k_r z_r (q) - l_1 p_1(q) - \cdots - l_s p_s(q)$. Furthermore, every irreducible component of $\calP(\calX/B; \mu)$ has dimension $\geq \dim B - g$. 
\end{theorem}

\begin{proof}
As explained in the proof of Theorem~\ref{thm:scheme-twisted}, it suffices to prove the result when the general fiber of $\calX$ is smooth. Since the problem is local on $B$, we assume that $B$ is affine. Take a relatively ample divisor $D$ on $\calX$ such that $D$ is contained in the smooth locus of $\pi$ and disjoint from the sections $z_i$ and $p_j$. Replacing $D$ with a high multiple of itself, we may assume that it intersects every component of a reducible fiber with high degree. Denote by $d$ the total degree of $D$ relative over $B$. 

Let $\omega_{\mu}$ be a twisted relative canonical line bundle on $\calX$ such that restricted to each fiber 
$(C, z_1, \ldots, z_r, p_1, \ldots, p_s)$ it is the unique twisted canonical line bundle of degree 
$\sum_{z_j\in C_i} k_j - \sum_{p_h\in C_i} l_h$ on every component 
$C_i$ of $C$. It follows that 
$$\pi_{*} \omega_{\mu}\left( \sum_{j=1}^s l_j p_j + D\right)$$ 
is a vector bundle of rank 
$$1-g+ \left(2g-2 + \sum_{j=1}^s l_j + d\right) = g-1 +d + \sum_{j=1}^s l_j $$ 
by Riemann-Roch. Let $\calP'$ be the corresponding projective bundle with fiber dimension $g-2 + d + \sum_{j=1}^s l_j $ over $B$. A point of $\calP'$ over $q' \in B$ is thus a section 
$\sigma \in H^0( \omega_{\mu}(D+ \sum_{j=1}^s l_j p_j)|_{X_{q'}})$, modulo the equivalence $\sigma \sim \lambda \sigma$ for a nonzero scalar $\lambda$. 

Consider the subvariety $\calP'(\calX/B; \mu)$ in $\calP'$ cut out by the following groups of equations: 
\begin{itemize}
\item \emph{Vanishing on $D$}.  We require that $\sigma$ vanishes on $D$. 
\item \emph{Ramification at $z_i$}. For each section $z_i$, $\sigma$ vanishes on $z_i$ with multiplicity $\geq k_i$. 
\end{itemize}
 
The vanishing condition on $D$ is given by $d$ equations. The ramification condition on each $z_i$ imposes $k_i$ equations, so in total they impose $\sum_{i=1}^r k_i = 2g-2 + \sum_{j=1}^s l_j$ equations. It follows that the dimension of every irreducible component of $\calP'(\calX/B; \mu)$ is at least 
\begin{displaymath}
 \dim B + \left(g-2 + d +\sum_{j=1}^s l_j\right) - d - \left(2g-2 + \sum_{j=1}^s l_j\right) = \dim B - g. 
 \end{displaymath}
Let $U$ be an irreducible component of $\calP'(\calX/B; \mu)$. For a general point $\sigma \in U$, there are two possibilities. First, if $\sigma$ vanishes on a component of the underlying curve $C$, then this property holds for all sections parametrized in $U$. The other case is that $\sigma$ has isolated zeros in 
$C$. Then it follows from construction that 
$$(\sigma)_0 = \sum_{i=1}^r k_i z_i + D \sim \omega_{\mu} + D + \sum_{j=1}^s l_j p_j, $$ 
hence $\sum_{i=1}^r k_i z_i -  \sum_{j=1}^s l_j p_j$ is a twisted meromorphic canonical divisor restricted to $C$. 
Conversely, a twisted meromorphic canonical divisor determines such a $\sigma$ up to scaling. By collecting irreducible components of $\calP'(\calX/B; \mu)$ of the latter type, we thus obtain the desired $\calP(\calX/B; \mu)$ parameterizing twisted meromorphic canonical divisors with signature $\mu$. 
\end{proof}

\begin{proposition}
\label{prop:meromorphic-smoothing}
In the above setting, suppose $(C, z_1, \ldots, z_r, p_1, \ldots, p_s)\in B$ is contained in an irreducible component $U$ of $\calP(\calX/B; \mu)$ such that $\dim U = \dim B - g$, i.e. the dimension bound is attained. Then $(C, z_1, \ldots, z_r, p_1, \ldots, p_s)\in \BPP(\mu)$. 
\end{proposition}

\begin{proof}
Let $\wt{\calX} \to \wt{B}$ be the versal family of curves around $(C, z_1, \ldots, z_r, p_1, \ldots, p_s)$. Let $f: B \to \wt{B}$ be the map 
locally inducing $\calX \to B$ with $n = r+s$ sections of marked points. Let $\wt{U}$ be a component of $\calP(\wt{\calX}/ \wt{B}; \mu)$ such that 
$U$ is a component of $f^{*}\wt{U}$ and let $\wt{C}$ be the point in $\wt{U}$ corresponding to the 
pointed curve $C$. By Theorem~\ref{thm:scheme-twisted-meromorphic} we have 
$$\dim \wt{U} \geq \dim \wt{B} - g.$$ 

If $\wt{U}$ does not lie entirely in the discriminant locus of $\wt{\calX}/\wt{B}$ parameterizing nodal curves, then it implies that a general point 
of $\wt{U}$ parameterizes an ordinary meromorphic canonical divisor with signature $\mu$ on a smooth curve, and hence we are done. Suppose on the contrary 
that $\wt{U}$ lies over a component $\wt{B}'$ of the discriminant locus. Note that $\wt{B}'$ 
is a hypersurface in $\wt{B}$, hence 
$$\dim \wt{U} \geq \dim \wt{B} - g > \dim \wt{B}' - g.$$
Thus every component of $f^{*}\wt{U}$, including $U$, has dimension $> \dim B - g$, contradicting the assumption 
that  $\dim U = \dim B - g$. 
\end{proof}

Now we consider twisted meromorphic canonical divisors on curves of compact type. Suppose $C$ is a pointed curve of compact type with $m+1$ irreducible components $C_0, \ldots, C_m$ of genus $g_0, \ldots, g_m$, respectively, and with $q_1, \ldots, q_{m}$ as the nodes. Let $z_1, \ldots, z_r, p_1, \ldots, p_s$ be the marked points in $C$ with respect to the signature $\mu = (k_1, \ldots, k_r, - l_1, \ldots, -l_s)$. Denote by 
$M_i$ the sum of zero orders minus pole orders for the zeros $z_{j_i}$ and poles $p_{h_i}$ contained in each $C_i$. Since $C$ is of compact type, 
there exists a unique twisted canonical line bundle whose restriction to each $C_i$ has degree $M_i$ and is of type 
$K_{C_i}(\sum_{j=1}^{m} t_{ij} q_j )$ for $t_{ij}\in \bbZ$. Recall that $C_i$ is a polar component if either $C_i$ contains some $p_j$, 
or contains some $q_j$ with $t_{ij} > 0$. 

\begin{theorem}
\label{thm:twisted-meromorphic}
In the above setting, if $C\in \BPP(\mu)$, then we have 
\begin{eqnarray}
\label{eq:twisted-meromorphic}
 \sum_{z_j\in C_i} k_j z_j - \sum_{p_h \in C_i} l_h p_h - \sum_{j=1}^{m} t_{ij} q_j \sim K_{C_i} 
\end{eqnarray}
for all $i$. Conversely if Relation~\eqref{eq:twisted-meromorphic} holds and further assume that all $C_i$ are polar components, then 
$C\in \BPP(\mu)$. 
\end{theorem}

\begin{proof}
The first part of the result follows from the same proof as Proposition~\ref{prop:limit-canonical}. Conversely, suppose Relation~\eqref{eq:twisted-meromorphic} holds. Let $B\subset \BM_{g,n}$ be the locus of nodal curves that are of the same topological type 
as $C$. It is an open dense subset of the corresponding boundary stratum of $\BM_{g,n}$, where $n=r+s$, hence 
$$\dim B = \dim \BM_{g,n} - m = 3g- 3 + n - m.$$ 
Consider the space of twisted 
meromorphic canonical divisors with signature $\mu$ over $B$. Restricted to each $C_i$, it is the stratum of ordinary meromorphic canonical divisors 
with zeros at $z_{j_i}$ and poles at $p_{h_i}$, and have zero or pole of order $-t_{ij}$ at $q_j$. By assumption $C_i$ is polar, hence it contains at least one pole. Therefore, using the dimensions of strata of meromorphic differentials, we have 
\begin{eqnarray*}
\dim \calP(\calX/B; \mu) & = & \sum_{i=0}^m (2g_i-3) + r + s + 2m \\
                                       & = & 2g -3 + n - m \\
                                       & = & \dim B - g.  
\end{eqnarray*}
The desired smoothability conclusion thus follows from Proposition~\ref{prop:meromorphic-smoothing}. 
\end{proof}

\begin{remark}
\label{rem:hodge-pesudo-meromorphic}
Analogous to Remark~\ref{rem:hodge-pseudo}, Theorem~\ref{thm:twisted-meromorphic} also holds for $C$ of pseudocompact type under the same assumption, by treating $K_{C_i}$ as the dualizing line bundle of $C_i$ and only taking the separating nodes $q_j$ into account.  
 Again, the first part follows from the explanation of Remark~\ref{rem:twisted-pseudo}, and the smoothability part follows from smoothing out all self-nodes by the plumbing operation to reduce to the case of compact type.  
\end{remark}

Note that Theorem~\ref{thm:main-meromoprhic-one-node} is a special case of the above result for curves of compact type with one node.  

\subsection{Curves of non-pseudocompact type}
\label{subsec:non-compact}

Let us study the closure of $\calP(\mu)$ in the locus of curves of non-pseudocompact type. It is possible to extend the notion of twisted canonical divisors to a curve of non-pseudocompact type, but we need to consider \emph{semistable models} of the curve by inserting chains of rational curves at a non-separating and external node, as already considered in 
the study of limit linear series (see e.g. \cite{EstevesMedeiros, OssermanNoncompact}). We explain this issue in detail in the following example. 

Some setup first. Let $\pi: \calX \to T$ be a one-parameter family of genus $g$ curves over a disk such that the generic fiber $C_t$ is smooth and the central fiber $C$ is nodal. We say that $\calX$ is a \emph{smoothing} of $C$. If $\calX$ is smooth everywhere but possibly at those nodes of $C$ that are self-intersections of each of its irreducible components, we say that $\calX$ is a \emph{regular smoothing}. The reason to consider 
regular smoothing families is that every irreducible component of $C$ is a \emph{Cartier} divisor on such $\calX$. In particular, if $C$ is of compact type, then every regular smoothing family is smooth everywhere. The reader can refer to \cite[Section 1.3]{EstevesMedeiros} for more details on this setting. 

Now, let $C_1$ be a nodal curve consisting of $X$ and $R$ union at two nodes $q_1$ and $q_2$, where $X$ is smooth of genus two and $R\cong \bbP^1$. Let $z$ be a marked point contained in $R$. Therefore, $(C_1, z)$ is a stable pointed curve parameterized in $\BM_{3,1}$. Furthermore, assume that $q_1$ is a 
Weierstrass point of $X$, i.e. $2q_1\sim K_X$. Let $C_3$ be the semistable model of $C_1$ by blowing up $q_1$ and inserting two exceptional $\bbP^1$-components, denoted by $E_1$ and $E_2$, see Figure~\ref{fig:noncompact}. 
 \begin{figure}[h]
    \centering
   \psfrag{C1}{$C_1$:}
    \psfrag{C3}{$C_3$:}
     \psfrag{E1}{$E_1$}
    \psfrag{E2}{$E_2$}
    \psfrag{X}{$X$}
     \psfrag{R}{$R$}
      \psfrag{z}{$z$}
    \psfrag{q1}{$q_1$}
     \psfrag{q2}{$q_2$}
    \includegraphics[scale=1.0]{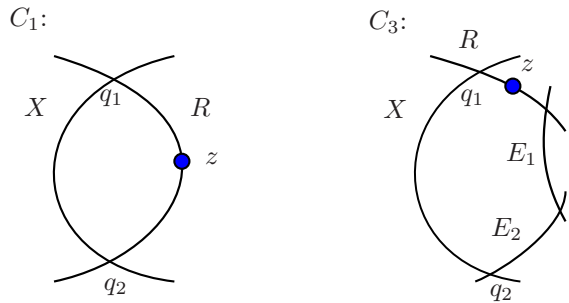}
       \caption{\label{fig:noncompact} A curve with a semistable model} 
 \end{figure}       

Here the subscript $i$ of $C_i$ stands for the length of the chain of rational curves 
between $q_1$ and $q_2$.  

\begin{proposition}
\label{prop:non-compact-example}
In the above setting, we have $(C_1, z) \in \BPP(4)^{\odd}$. In addition, any one-dimensional smoothing family of $(C_1, z)$ into $\calP(4)$ must be 
singular. However, there exists a one-dimensional regular smoothing family of $(C_3, z)$ into $\calP(4)$. 
\end{proposition}

\begin{proof}
If we forget $z$ and blow down $R$, then we get a one-nodal curve $C_0$ with a node $q$. By assumption, $C_0$ 
is not in the closure of genus three hyperelliptic curves, hence it admits a canonical embedding as a plane nodal quartic (see \cite[Proposition 2.3]{HassettQuartics}). The section $3q_1 + q_2$ of the dualizing line bundle $K_{X}(q_1 + q_2)$ corresponds to a line $L$ in $\bbP^2$ that has total contact order four to $C_0$ at $q$, that is, contact order three to the branch of $q_1$ and transversal to the branch of $q_2$, see Figure~\ref{fig:nodal-quartic}. 

 \begin{figure}[h]
    \centering
   \psfrag{C}{$C_0$}
    \psfrag{q}{$q$}
     \psfrag{L}{$L$}
      \includegraphics[scale=0.8]{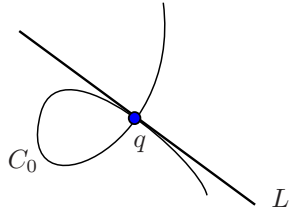}
       \caption{\label{fig:nodal-quartic} A nodal quartic with a hyperflex at the node} 
 \end{figure}       

Note that the space of plane quartics that have contact order four to $L$ at $q$ is irreducible, and its general element corresponds to a smooth, non-hyperelliptic curve $C$ with a hyperflex $q$, i.e. $4q \sim K_C$, and hence $(C, q) \in  \BPP(4)^{\odd}$. Blowing up $C_0$ at $q$ yields the stable pointed curve $(C_1, z)$, thus proving the first part of the proposition. 

Next, suppose there is a smoothing family $\pi: \calX \to T$ such that $\calX$ is everywhere smooth with a section $Z$, $C_1$ is the special fiber with $Z(0) = z$, and the generic fiber $(C_t, Z(t)) \in \calP(4)$. Let $\omega_{\pi}$ be the relative dualizing line bundle of the family. 
Consider the direct image sheaf 
$$ \calF = \pi_{*}(\omega_{\pi}(2X - 4Z)). $$
By assumption, $\calF|_{C_t} = H^0(C_t, \OO) = \bbC$, hence $\calF|_{C_1}$ must have a section not identically zero. 
Note that 
$$\omega_{\pi}(2X - 4Z)|_{X} = K_X(-q_1 - q_2), \quad \omega_{\pi}(2X - 4Z)|_{R} = \OO_{\bbP^1}. $$
It implies that $q_1 + q_2 \sim K_X$, contradicting the assumption that $2q_1\sim K_{X}$. 

Finally, fix $q\in L \subset \bbP^2$ in the first paragraph. Among all plane quartics that have contact order four to $L$ at $q$, take a general one and denote it by $Q$. Consider the pencil of plane quartics generated by $tQ+C_0$, where $t$ is the base parameter. Locally around $(q=(0,0), t=0)$ this family can be written as  
$t(y-x^4) + x(y-x^3)$, hence has an $A_3$-singularity at the origin. Running stable reduction by resolving it as well as other base points of the pencil, it is easy to see that $(C_3, z)$ arises as the central fiber in the resulting everywhere regular family (see also
the proof of Theorem~\ref{thm:double-conic} for a concrete and harder example of stable reduction). 
\end{proof}

\begin{remark}
\label{rem:non-compact-example}
In the above proof, when theta characteristics $\eta_t$ on smooth curves $C_t \in \calP(4)^{\odd}$ degenerate to 
$C_0$, it is not hard to see that the limit theta characteristic $\eta_0$ is of the second kind (see Section~\ref{subsec:spin}), i.e. it consists of 
$(\OO_X(q_1), \OO_{R}(1))$ on the semistable model $C_1$ of $C_0$ without the marked point. The parity of 
$\eta_0$ equals $h^0(X, q_1) = 1$, hence it is odd, as already predicted by Proposition~\ref{prop:non-compact-example}. 
See Remark~\ref{rem:hyp-odd-spin} for a comparable example. 
\end{remark}

The essence of Proposition~\ref{prop:non-compact-example} says that in order to see the containment of $(C_1, z)$ in $\BPP(4)$, it is necessary to insert chains of rational curves at a non-separating and external node, and consider semistable models. For a curve $C$ of pseudocompact type, after this procedure it remains to be of pseudocompact type, and in Section~\ref{subsec:twisted} we have seen that the twisting operation can provide a unique twisted canonical line bundle that has the desired degree $M_i$ on each component of $C$ and has degree zero on the exceptional $\bbP^1$-components. However, for curves of non-pseudocompact type, inserting extra rational chains at non-separating and external nodes can change the twisting pattern significantly. 

Let us study this issue in general. Suppose $C$ is a pointed nodal curve (possibly semistable) with $n = r+s$ marked points and 
$m$ irreducible components $C_1, \ldots, C_m$. Denote by $a_{ij}$ the number of nodes contained in $C_i\cap C_j$ for $i\neq j$. Set $a_{ii} = - \sum_{j\neq i} a_{ij}$ for all $i$. The $m\times m$ matrix $L(C) = (a_{ij})$ (or sometimes $-L(C)$) is called the \emph{Laplacian matrix} of the dual graph of $C$. In particular, $L(C)$ is symmetric, zero is an eigenvalue of $L(C)$, and $(1, \ldots, 1)^t$ is an eigenvector associated to zero. 

Now take a signature $\mu = (k_1, \ldots, k_r, - l_1, \ldots, -l_s)$. As before, let $M_i$ be the sum of zero orders minus the pole orders  
for those marked points contained in $C_i$. 

\begin{proposition}
\label{prop:laplacian}
In the above setting, if there is a regular smoothing family $\pi: \calX \to T$ of $C$ into $\calP(\mu)$, then there exists an integer sequence $(b_1, \ldots, b_m)$, unique modulo $(1, \ldots, 1)$, such that 
\begin{eqnarray}
\label{eq:laplacian}
L(C) \cdot (b_1, \ldots, b_m)^t = (M_1- 2g_1 + 2, \ldots, M_m- 2g_m+2)^t. 
\end{eqnarray}
Moreover, $\sum_{i=1}^r k_i z_i - \sum_{j=1}^s l_j p_j$ is a twisted meromorphic canonical divisor with respect to the twisted canonical line bundle 
$\omega_{\pi}(\sum_{i=1}^m b_i C_i)|_C$. 
\end{proposition}

\begin{proof}
Let $Z_i$ and $P_j$ denote the sections of $z_i$ and $p_j$ in $\calX$, respectively. 
Since $\calX$ is a regular smoothing, every component $C_i$ is a Cartier divisor of $\calX$. Note that $\sum_{i=1}^r k_i Z_i-\sum_{j=1}^s l_j P_j$
restricted to the open locus $\calX^{\circ} = \calX\backslash C$, as a meromorphic section of $\pi_{*}\omega_{\pi}|_{\calX^{\circ}}$, 
extends uniquely to $\sum_{i=1}^r k_i Z_i - \sum_{j=1}^s l_j P_j - \sum_{i=1}^m b_i C_i$ as a meromorphic section of $\pi_{*}\omega_{\pi}$ for some $b_i \in \bbZ$. Therefore, restricted to the central fiber $C$ we conclude that $\sum_{i=1}^r k_i z_i - \sum_{j=1}^s l_j p_j$ 
is a twisted canonical divisor with respect to the twisted canonical line bundle $K_C\otimes \OO_{\calX}(\sum_{i=1}^m b_i C_i)_{C}$. Moreover, 
$\deg \OO_{\calX}(C_j)|_{C_i} = a_{ij}$, hence Relation~\eqref{eq:laplacian} follows by taking the degree on each $C_i$. Finally, the uniqueness
of $(b_1,\ldots, b_m)$ modulo $(1, \ldots, 1)$ follows from the fact that $\OO_{\calX}( \sum_{i=1}^m d_i C_i)$ is trivial if and only if 
$(d_1, \ldots, d_m)$ is a multiple of $(1, \ldots, 1)$. Alternatively, it also follows from the fact that the algebraic multiplicity of the zero eigenvalue is one for the Laplacian matrix of any connected graph (see e.g. \cite[Lemma 1.7]{Chung}), which is our case because $C$ is connected. 
\end{proof}

\begin{corollary}
\label{cor:laplacian}
If a pointed stable nodal curve $C'$ is contained in $\BPP(\mu)$, then $C'$ admits a semistable model $C$ by inserting chains of rational curves at nodes of $C$ such that there exists a regular smoothing family $\pi: \calX \to T$ of $C$ into $\calP(\mu)$, and hence $C$ satisfies the properties 
described in Proposition~\ref{prop:laplacian}.  
\end{corollary}

\begin{proof}
By assumption, there exists a smoothing family of $C'$ into $\BPP(\mu)$. If it is not regular, running semistable reduction explained in 
\cite[Section 2.7]{EstevesMedeiros} yields the desired regular smoothing family with a semistable model $C$ as the new central fiber. The rest 
part of the corollary follows from Proposition~\ref{prop:laplacian}.  
\end{proof}

In the following example we illustrate how to apply the above results. 

\begin{example}
Recall the semistable model $(C_3, z)$ of the stable curve $(C_1, z)$ described in Proposition~\ref{prop:non-compact-example}. 
Line up the components of $C_3$ in the order $X, R, E_1$ and $E_2$. Then the Laplacian matrix of $C_3$ is 
$$L(C_3) = \left( \begin{matrix} -2 & 1 & 0 &  1 \\ 
 1 & -2 & 1 &  0 \\
  0 & 1 & -2 &  1 \\
   1 & 0 & 1 &  -2 
\end{matrix} \right). $$
Since $\deg K_{C_3}|_X = 4$ and $\deg K_{C_3}|_{R} = \deg K_{C_3}|_{E_i} = 0$ and we want the twisted canonical line bundle to have degree four on $R$ and degree zero on the other components, one checks that 
$$ L(C_3) \cdot (0, -3, -2, -1)^t = (-4, 4, 0, 0)^t. $$
Therefore, for a regular smoothing family $\pi: \calX \to T$ of $C_3$ 
$$\calL = \omega_{\pi}(-3R - 2E_1 - E_2)$$ 
is the twisted relative canonical line bundle whose restriction to $C_3$ has the desired degree on each component. Note that 
$$\calL|_{X} = K_X(-2q_1), \quad \calL|_{R} = \OO_{\bbP^1}(4), \quad \calL|_{E_i} = \OO_{\bbP^1}. $$
If $4z$ is a section of $\calL|_{C_3}$, then $\calL|_{C_3}(-4z)$ has degree zero on every component of $C_3$ and has a section not identically zero, hence it is trivial restricted to every component of $C_3$. It follows that $2q_1 \sim K_X$ and $q_1$ is a Weierstrass point of $X$. This 
provides an inverse statement to Proposition~\ref{prop:non-compact-example}, that is, if such $(C_3, z)$ can be smoothed into $\calP(4)$ via 
a regular smoothing family, then $q_1$ is a Weierstrass point of $X$. 
\end{example}

\section{Boundary of hyperelliptic and spin components}
\label{sec:spin}

Let us first prove Theorem~\ref{thm:spin}. We treat the cases of holomorphic and meromorphic differentials simultaneously. 

\begin{proof}[Proof of Theorem~\ref{thm:spin}]
Let $\mu = (2k_1, \ldots, 2k_r, -2l_1, \ldots, -2l_s)$ be a partition of $2g-2$ with $k_i, l_j \in \bbZ^{+}$. If $s=0$, it is just 
a signature of holomorphic differentials. Let $\calP(\mu)^{\odd}$ and $\calP(\mu)^{\even}$ be the two spin components of $\calP(\mu)$. 
Prove by contradiction. Suppose $(C, z_1, \ldots, z_r, p_1, \ldots, p_s)\in \BM_{g,n}$ with $n=r+s$ is a curve of pseudocompact type contained in 
both $\BPP(\mu)^{\odd}$ and $\BPP(\mu)^{\even}$. 
Let $\pi: \calX\to T$ be a smoothing family of spin curves $(C, z_1(t), \ldots, z_r(t), p_1(t), \ldots, p_s(t))$ in $\calP(\mu)$ degenerating to the pointed curve $C$. Let $Z_i$ and $P_j$ be the sections corresponding to $z_i(t)$ and $p_j(t)$, respectively. 
Denote by 
$$\eta_t = \OO_{C_t}\left(\sum_{i=1}^r k_i z_i(t) - \sum_{j=1}^s l_j p_j(t)\right)$$
the theta characteristic on a generic fiber $C_t$. 

To describe the limit spin structure when $\eta_t$ degenerates to $C$, as described 
in Section~\ref{subsec:spin} we blow up each separating node of $C$ to insert an exceptional $\bbP^1$-component. The resulting curve remains to be of pseudocompact type, and we still denote by $C$ the special fiber. In particular, the marked points $z_j$ and $p_h$ are contained in the non-exceptional components of $C$ only.  
Let $C_1, \ldots, C_m$ be the irreducible components of $C$, and let $q_1, \ldots, q_{m-1}$ be the \emph{separating}
nodes of $C$. For each non-exceptional component $C_i$, let $\eta_i$ be the 
 limit spin structure restricted to $C_i$, and on an exceptional $\bbP^1$ it is $\OO(1)$ (see Section~\ref{subsec:spin}). Also denote by $\eta_{\calX}$ the universal theta characteristic line bundle such that $\eta_{\calX}|_{C_t} = \eta_t$ and $\eta_{\calX}|_{C} = \eta$. 
Since $C$ is of pseudocompact type and $\sum_{j=1}^r k_j - \sum_{h=1}^s l_h = g-1 = \deg \eta_t$, 
there exists a unique twisted universal theta characteristic  
$$\eta_{\mu} = \eta_{\calX}\left(\sum_{i=1}^m b_i C_i\right)$$ 
for some $b_i \in \bbZ$, independent of $\calX$ and depending on $C$ only, 
such that 
$$\deg \eta_{\mu}|_{C_i} = \sum_{z_j\in C_i} k_j - \sum_{p_h\in C_j} l_h $$ 
for non-exceptional components $C_i$ and $\deg \eta_{\mu}|_{C_j} = 0$ when $C_j$ is exceptional. Denote by 
$$\calL = \eta_{\mu} \left( - \sum_{j=1}^r k_j Z_j + \sum_{h=1}^s l_h P_h\right) $$ 
and consider the direct image sheaf $\calF = \pi_{*} \calL$. 
By assumption, $\calF|_{C_t} = H^0(C_t, \OO) = \bbC$, hence 
$\calF|_{C} = H^0(C, \calL|_C)$ has a section not identically zero. Since the degree of $\calL$ restricted to every component of $C$ is zero and $C$ is connected, it implies that $\calL$ is the trivial line bundle, and hence 
$$ \eta_{\mu}|_{C_i} = \OO_{C_i}\left( \sum_{z_j\in C_i} k_jz_j - \sum_{p_h\in C_j} l_h p_h\right). $$
It follows that 
$$ \eta_i = \eta_{\calX}|_{C_i} = \eta_{\mu}\left(-\sum_{i=1}^m b_i C_i\right)|_{C_i}, $$
which is independent of $\calX$ and depends on $C$ only. 
Thus the parity of the limit spin structure $\eta$ on $C$ is given by $\sum_{i=1}^m h^0(C_i, \eta_i) \pmod{2}$, and hence cannot be both even and odd, leading to a contradiction. 
\end{proof}

\begin{remark}
The above proof implies that if an irreducible component $C_i$ has self-nodes, then $\eta_i$ is of the first kind (see Section~\ref{subsec:spin}), i.e. $\eta_i^{\otimes 2} = K_{C_i}$, where $K_{C_i}$ is regarded as the dualizing line bundle of $C_i$. In other words, the zeros and poles do not degenerate to a self-node, so we do not blow up a self-node to insert an exceptional $\bbP^1$, and hence $\eta_i$ cannot be of the second kind. 

Moreover, for the signature $(2k_1, \ldots, 2k_r, -1, -1)$ with $k_i > 0$, 
the corresponding stratum of meromorphic differentials 
also has two spin components (see \cite[Section 5.3]{Boissy}). Suppose $X$ is a pointed smooth curve contained in this stratum with $p_1$ and $p_2$ as the two simple poles. Identifying $p_1$ and $p_2$, we obtain an irreducible one-nodal curve $X'$. Let 
$\eta' = \OO_{X'}( \sum_{j=1}^r k_jz_j)$. 
Since $\eta'^{\otimes 2} = K_{X'}$, $\eta'$ is a spin structure of the first kind on $X'$, 
and the parity of $\eta'$ determines that of $X$. Now suppose $X$ degenerates to a pointed curve $C$ of pseudocompact type. As long as $p_1$ and $p_2$ 
are contained in the same irreducible component of $C$, identifying them as a node yields a curve $C'$, which is still of pseudocompact type. Therefore, 
keeping track of $\eta'$ on $X'$ degenerating to $C'$, the same proof as above goes through, and hence the closures of the two spin components of 
$\HH(2k_1, \ldots, 2k_r, -1, -1)$ remain disjoint in the locus of such pointed nodal curves.  
\end{remark}

\begin{example}
Let us explicitly determine the limit spin structure for a pointed one-nodal curve $C = C_1 \cup_q C_2$ contained in $\BPP(\mu)$ 
for $\mu = (2k_1, \ldots, 2k_r, -2l_1, \ldots, -2l_s)$, which can help the reader capture the upshot of the above proof. 
Blow up $q$ to insert an exceptional $\bbP^1$-component between $C_1$ and $C_2$. Let $q_i = C_i \cap \bbP^1$. 
Let $\eta = (\eta_1, \eta_2, \OO(1))$ be the limit spin structure, where $\OO(1)$ is the unique degree one line bundle on $\bbP^1$ and $\eta_i$ is an ordinary theta characteristic on $C_i$. For $i=1,2$, denote by 
$$N_i = \sum_{z_j\in C_i} k_j - \sum_{p_h\in C_i} l_h. $$
Then we have 
$$ \eta_i = \OO_{C_i}\left(\sum_{z_j\in C_i} k_jz_j - \sum_{p_h\in C_i} l_h p_h + (g_i-1-N_i)q_i\right)$$ 
for $i=1,2$. 
Hence the parity of $(C, \eta)$ is given by $ h^0(C_1, \eta_1) + h^0(C_2, \eta_2) \pmod{2}$. 
\end{example}

In general, if we consider all possible types of pointed stable curves, then 
 $\BPP(\mu)^{\hyp}$, $\BPP(\mu)^{\odd}$ and $\BPP(\mu)^{\even}$ can intersect, illustrated by the following result. 
 
Consider the stratum $\calP(4)$ in genus three. It consists of two connected components $\calP(4)^{\hyp}$ and $\calP(4)^{\odd}$. In this case the hyperelliptic component coincides with the even spin component. Let $C$ be the union of $X$ and $\bbP^1$, attached at two nodes $q_1$ and $q_2$, where $X$ is a smooth curve of genus two, $q_1$ and $q_2$ are conjugate under the hyperelliptic involution of $X$, and $\bbP^1$ contains a marked point $z$, see Figure~\ref{fig:banana}. 
\begin{figure}[h]
    \centering
    \psfrag{X}{$X$}
     \psfrag{P}{$\mathbb P^1$}
      \psfrag{q1}{$q_1$}
      \psfrag{q2}{$q_2$}
      \psfrag{z}{$z$}
    \includegraphics[scale=0.6]{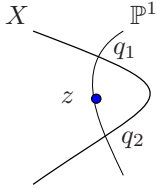}
    \caption{\label{fig:banana} A stable curve of non-pseudocompact type}
\end{figure} 

In particular, as a pointed curve $(C, z)$ is stable and is not of pseudocompact type.   
  
 \begin{theorem}
 \label{thm:double-conic}
 $\BPP(4)^{\hyp}$ and $\BPP(4)^{\odd}$ intersect in $\BM_{3,1}$, both containing $(C, z)$. 
 \end{theorem}

\begin{proof}
Since $q_1$ and $q_2$ are hyperelliptic conjugate in $X$, it implies that $C$ admits an admissible double cover with $z$ as a ramification point, and hence $(C, z)$ is contained in $\BPP(4)^{\hyp}$, see Figure~\ref{fig:doublecover}. 
\begin{figure}[h]
    \centering
    \psfrag{z}{$z$}
    \psfrag{X}{$X$}
    \psfrag{q1}{$q_1$}
    \psfrag{q2}{$q_2$}
    \psfrag{P1}{$\mathbb{P}^1$}
    \includegraphics[scale=0.8]{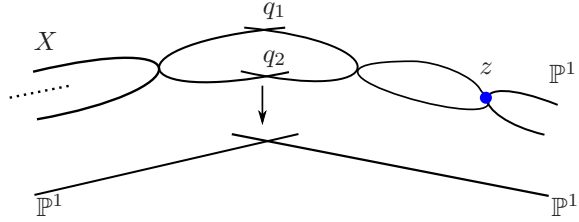}
    \caption{\label{fig:doublecover} A degenerate hyperelliptic double cover}
\end{figure}

To see the containment in $\BPP(4)^{\odd}$, take a pencil of plane quartics generated by $Q_0$ and $Q_1$, where 
$Q_0$ is a double conic tangent to a line $L$ at $z$, and $Q_1$ is general among all plane quartics that are tangent to $L$ at $z$ with multiplicity four, see Figure~\ref{fig:quartics}. 
\begin{figure}[h]
    \centering
    \psfrag{z}{$z$}
    \psfrag{Q0}{$Q_0$}
    \psfrag{Q1}{$Q_1$}
    \psfrag{L}{$L$}
    \includegraphics[scale=0.8]{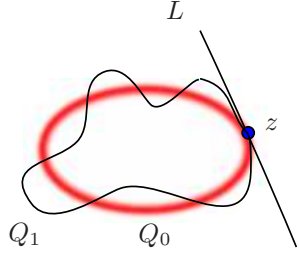}
    \caption{\label{fig:quartics} Quartics degenerate to a double conic with a common hyperflex}
\end{figure}

Let us carry out \emph{stable reduction} to this pencil to see that the stabilization of $(Q_0, z)$ is 
$(C, z)$. Suppose the pencil is given by $(y-x^2)^2 + t(x^4 + yf(x,y))$ with $t$ the base parameter, where $f(x,y)$ is general of degree three. All curves in this family share a common hyperflex line defined by $y = 0$. Let $Z$ be the section of hyperflex points given by 
$x=y=0$. Let $C$ 
be the reduced central fiber when $t=0$, i.e. the double conic is $2C$. At 
$x=y=t=0$, the total family as a surface has an $A_3$-singularity, i.e. locally of type $u v = w^4$. First, blowing it up reduces the singularity to of type $A_1$, and the proper transform of $C$ meets the new singularity, where the two exceptional curves $E_1$ and $E_2$ meet, and $Z$ meets exactly one of $E_1$ and $E_2$. Blowing up the singularity again makes the family smooth, and now $C$ meets the interior of the new exceptional curve $E_3$. The pencil has six other base points, which are all $A_1$-singularities. Blowing them up yields six exceptional curves $B_1,\ldots, B_6$. All the $E_i$ and $B_j$ have self-intersection $-2$. 
Let $b_j = B_j\cap C$, $e_i = E_i\cap E_3$ for $i=1,2$, and $e_3 = E_3\cap C$. At this stage the central fiber of the family consists of 
$$2C + E_1 + E_2 + 2E_3 + \sum_{j=1}^6 B_j,$$
see Figure~\ref{fig:blowup1}. 

\begin{figure}[h]
    \centering
    \psfrag{z}{$z$}
    \psfrag{B1}{$B_1$}
    \psfrag{B6}{$B_6$}
     \psfrag{E1}{$E_1$}
    \psfrag{E2}{$E_2$}
     \psfrag{E3}{$E_3$}
    \psfrag{b1}{$b_1$}
     \psfrag{b6}{$b_6$}
    \psfrag{e1}{$e_1$}
     \psfrag{e2}{$e_2$}
    \psfrag{e3}{$e_3$}
    \psfrag{C}{$C$}
    \includegraphics[scale=0.8]{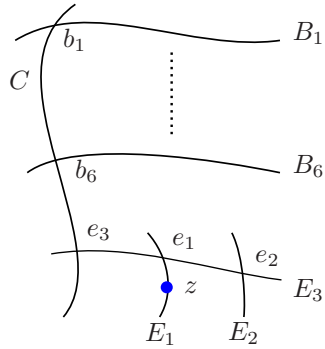}
    \caption{\label{fig:blowup1} The central fiber after resolving the singularities}
\end{figure}

Take a base change of degree two branched along all $B_j$, $E_1$ and $E_2$. In particular, $b_j$ and $e_1, e_2$ are branch points on $C$ and on $E_3$, respectively. Pull back the central fiber and divide it by two. We obtain the new central fiber reduced and consisting of 
$$ \wt{C} + \wt{E}_1 + \wt{E}_2 + \wt{E}_3 + \sum_{j=1}^6 \wt{B}_j,$$
see Figure~\ref{fig:blowup2}. 

\begin{figure}[h]
    \centering
    \psfrag{z}{$z$}
    \psfrag{B1}{$\wt{B}_1$}
    \psfrag{B6}{$\wt{B}_6$}
     \psfrag{E1}{$\wt{E}_1$}
    \psfrag{E2}{$\wt{E}_2$}
     \psfrag{E3}{$\wt{E}_3$}
    \psfrag{q1}{$q_1$}
     \psfrag{q2}{$q_2$}
        \psfrag{C}{$\wt{C}$}
    \includegraphics[scale=0.8]{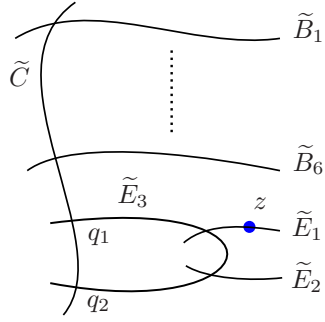}
    \caption{\label{fig:blowup2} The central fiber after a base change}
\end{figure}

Here $\wt{C}$ is a double cover of $C$ branched at all $b_j$, hence of genus two, $\wt{E}_3$ is a double cover of $E_3$ branched at $e_1, e_2$, 
$\wt{E}_3^2 = -4$, and $\wt{B}_j^2 = \wt{E}_i^2 = -1$ for $i=1,2$ and for all $j$. Moreover, $\wt{E}_3$ meets $\wt{C}$ at two points $q_1, q_2\in \wt{C}$ 
that are the inverse images of $e_3$ under the double cover, hence $q_1$ and $q_2$ are conjugate under the hyperelliptic involution of $\wt{C}$. 
Finally, blowing down $\wt{E}_1, \wt{E}_2$ and $B_j$, we obtain the desired curve configuration as $(C, z)$. 
Since $Z$ meets exactly one of $E_1$ and $E_2$, after blowing down $E_1$ and $E_2$, $Z$ only meets $\wt{E}_3$ in the central fiber, and hence as a marked curve $\wt{E}_3$ along with $q_1, q_2$ and $z$ is stable. 
\end{proof}

\begin{remark}
\label{rem:hyp-odd-spin}
In the above setting, let $C_0$ be the stable model of $C$ after forgetting $z$, i.e. $C_0$ is obtained by identifying 
$q_1$ and $q_2$ in $X$ as a node. When theta characteristics $\eta_t = \OO_{C_t}(2z)$ degenerate to $C_0$ from smooth curves 
$C_t \in \calP(4)$, it is not hard to see that the limit spin structure $\eta_0$ is of the first kind, whose pullback $\eta_0'$ 
on $X$ is $\eta'_0 = \OO_X(q_1 + q_2) = K_X$. There are two ways to identify fibers of $\eta'_0$ over $q_1$ and $q_2$ such that it descends to $\eta_0$ on $C_0$ with $h^0(C_0, \eta_0) = 1$ and $h^0(C_0, \eta_0) = 2$, respectively, hence both parities can appear. 
\end{remark}

\section{Weierstrass point behavior of general differentials}
\label{sec:weierstrass}

Suppose $\calP(\mu)$ is a stratum (component) of effective canonical divisors with signature $\mu = (m_1, \ldots, m_n)$. If $m_1 \geq g$, then by Riemann-Roch, $$ h^0(C, gz_1) = h^0(C, (m_1-g)z_1 + m_2 z_2 + \cdots + m_n z_n) + 1 \geq 2, $$
hence by definition $z_1$ is a Weierstrass point for all $(C, z_1, \ldots, z_n)\in \calP(\mu)$. Now assume that $m_1 < g$. Then it is natural to expect that for a \emph{general} 
$(C, z_1, \ldots, z_n)\in \calP(\mu)$, $z_1$ is \emph{not} a Weierstrass point. Similarly for a stratum (component) of meromorphic canonical divisors $\calP(\mu)$ with $\mu = (k_1, \ldots, k_r, -l_1, \ldots, -l_s)$, it is natural to expect that for a \emph{general} $(C, z_1, \ldots, z_r, p_1, \ldots, p_s)\in \calP(\mu)$, $z_1$ and $p_1$ are \emph{not} Weierstrass points. 

We remark that for hyperelliptic components, the answer is straightforward. First consider hyperelliptic components of holomorphic differentials
(see \cite[Definition 2]{KontsevichZorich}). For $(C, z)\in \calP(2g-2)^{\hyp}$, we have seen that $z$ is a 
Weierstrass point of $C$. For  
$(C, z_1, z_2)\in \calP(g-1, g-1)^{\hyp}$, the two zeros $z_1$ and $z_2$ are conjugate under the hyperelliptic involution, hence they are not Weierstrass points. Now consider hyperelliptic components of meromorphic differentials (see \cite[Proposition 5.3]{Boissy}). 
For $(C, z, p)\in \calP(2k, -2l)^{\hyp}$, both $z$ and $p$ are Weierstrass points. For $(C, z, p_1, p_2)\in \calP(2k, -l, -l)^{\hyp}$, 
$z$ is a Weierstrass point, but $p_1$ and $p_2$ are conjugate under the hyperelliptic involution, hence they are not. 
For $(C, z_1, z_2, p)\in \calP(k, k, -2l)^{\hyp}$, $p$ is a Weierstrass point, but $z_1$ and $z_2$ are conjugate under the hyperelliptic involution, hence they are not. Finally for $(C, z_1, z_2, p_1, p_2)\in \calP(k, k, -l, -l)^{\hyp}$, $z_1$ and $z_2$ (resp. $p_1$ and $p_2$) are conjugate under the hyperelliptic involution, hence all of them are not Weierstrass point. Therefore, for the rest of the section when we speak of a stratum (component), we assume that it is not hyperelliptic. 

Below we explore several approaches and establish desired results in a number of cases. 

\subsection{Merging zeros and poles}
\label{subsec:merging}

First, we point out that the case of $\calP(m_1, \ldots, m_n)$ can be reduced to $\calP(m_1, 2g-2-m_1)$. 

\begin{lemma}
\label{lem:weierstrass-two}
For $m_1 < g$, if $z_1$ in a general curve $(C, z_1, z_2)\in \calP(m_1, 2g-2-m_1)$ is not a Weierstrass point, then neither is $z_1$ in a general curve 
$(C, z_1, \ldots z_n)\in \calP(m_1, \ldots, m_n)$. 
\end{lemma}

\begin{proof}
By \cite{KontsevichZorich}, a zero of order 
$2g-2 - m_1 = m_2 + \cdots + m_n$ 
can be split off as distinct zeros of order $m_2, \ldots, m_n$, respectively. In other words, marking $z_1$ and lifting the strata into $\BM_{g,1}$, then $\calP(m_1, 2g-2-m_1)$ is contained in the closure of $\calP(m_1, \ldots, m_n)$. If a 
general curve $(C, z_1, \ldots z_n)\in \calP(m_1, \ldots, m_n)$ has $z_1$ as a Weierstrass point, then the closure of the locus of Weierstrass points in $\BM_{g,1}$ contains $\calP(m_1, \ldots, m_n)$, and hence contains $\calP(m_1, 2g-2-m_1)$, contradicting the assumption that $z_1$ in a general curve $(C, z_1, z_2)\in \calP(m_1, 2g-2-m_1)$ is not a Weierstrass point. 
\end{proof}

Next, for the case of meromorphic differentials. There is a similar procedure of merging poles. 

\begin{lemma}
\label{lem:merging-pole}
A pole of order $d = d_1 + d_2$ can be split off as two poles of order $d_1$ and $d_2$, respectively, i.e. a meromorphic differential 
in $\HH(k_1, \ldots, k_r, -d, -l_1, \ldots, -l_s)$ is a limit of differentials in $\HH(k_1, \ldots, k_r, -d_1, -d_2, -l_1, \ldots, -l_s)$. 
\end{lemma} 

\begin{proof}
Let $p$ be a pole of order $d$. As we have seen, the local flat-geometric neighborhood of $p$ can be constructed by gluing $d-1$ pairs of basic domains $D_i^{\pm}$, where $D_i^{+}$ has boundary rays $a_i$ and $l_i$, and $D_i^{-}$ has boundary rays $a_i$ and $l_{i+1}$, see Figure~\ref{fig:pole-basic}. 

\begin{figure}[h]
    \centering
    \psfrag{ai}{$a_i$}
    \psfrag{li}{$l_i$}
    \psfrag{li+1}{$l_{i+1}$}
     \psfrag{D+}{$D_i^{+}$}
    \psfrag{D-}{$D_i^{-}$}
    \includegraphics[scale=0.8]{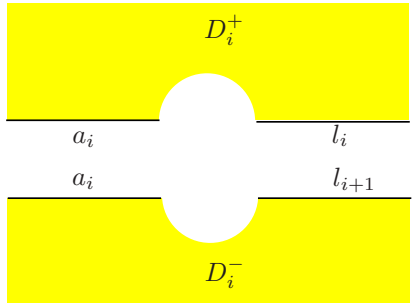}
    \caption{\label{fig:pole-basic} A pair of broken half-planes as basic domains}
\end{figure}

Now truncate $a_1$ in $D_1^{+}$ and $l_{d_1+1}$ in $D_{d_1+1}^{+}$ by a vertical half-line $x$ upward. Truncate $a_1$ in $D_1^{-}$ and $l_{d_1+1}$ in $D_{d_1}^{-}$ by a vertical half-line $y$ downward. We choose the truncated line segments in $a_1$ (resp. in $l_{d_1+1}$) such that they can still be glued under translation, see Figure~\ref{fig:pole-truncate}. 

\begin{figure}[h]
    \centering
    \psfrag{a1}{$a_1$}
    \psfrag{l1}{$l_1$}
    \psfrag{l2}{$l_{2}$}
     \psfrag{D1+}{$D_1^{+}$}
    \psfrag{D1-}{$D_1^{-}$}
    \psfrag{x}{$x$}
    \psfrag{y}{$y$}
    \psfrag{Dd+}{$D_{d_1+1}^{+}$}
     \psfrag{Dd-}{$D_{d_1}^{-}$}
    \psfrag{ad+}{$a_{d_1+1}$}
     \psfrag{ad}{$a_{d_1}$}
    \psfrag{ld+}{$l_{d_1+1}$}
    \psfrag{ld-}{$l_{d_1+1}$}
  \includegraphics[scale=0.8]{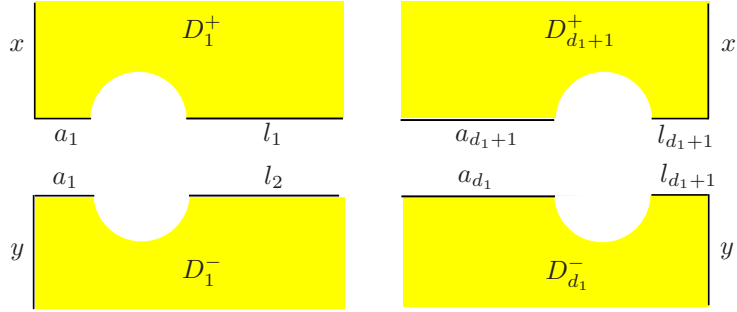}
    \caption{\label{fig:pole-truncate} Basic domains after truncation}
\end{figure}

After this operation, $D_2^{-}, \ldots, D_{d_1-1}^{-}, D_{d_1}^{-}\cup D_{1}^{-}$ and $D_2^{+}, \ldots, D_{d_1-1}^{+}, D_{1}^{+}$ form a flat-geometric neighborhood of a pole of order $d_1$. Similarly, the remaining basic domains form a pole of order $d_2$. If we reverse this process by extending the line segments in $a_1$ and in $l_{d_1+1}$ arbitrarily long, the two poles thus merge to a single pole of order $d_1 + d_2 = d$. 
\end{proof}

\begin{remark}
Note that one cannot always merge a zero and a pole. For instance, differentials in $\HH(2, -1, -1)$ cannot specialize to a differential in $\HH(1,-1)$, because differentials of the latter type do not exist. 
\end{remark}

We can similarly reduce the case $\calP(k_1, \ldots, k_r, -l_1, \cdots, -l_s)$ to $\calP(k_1, \ldots, k_r, -l)$, where $l = \sum_{i=1}^s l_1 + \cdots + l_s$. 

\begin{lemma}
\label{lem:weierstrass-two-meromorphic}
If $p$ in a general curve $(C, z_1, \ldots, z_r, p)\in \calP(k_1, \ldots, k_r, -l)$ is not a Weierstrass point, neither is $p_i$ in 
a general curve $(C, z_1, \ldots z_r, p_1, \ldots, p_s)\in \calP(k_1, \ldots, k_r, -l_1, \cdots, -l_s)$. 
\end{lemma}

\begin{proof}
The same proof as in Lemma~\ref{lem:weierstrass-two} works, with merging poles instead of zeros 
guaranteed by Lemma~\ref{lem:merging-pole}. 
\end{proof}

\subsection{Curves with an elliptic tail}
\label{subsec:tail}

As already used in \cite{EisenbudHarrisWeierstrass} and \cite{Bullock}, degenerating to nodal curves with an elliptic tail provides a powerful method to study the geometry of general curves. 

\begin{proposition}
\label{prop:weierstrass-holomorphic}
Suppose $(X, z_1, \ldots, z_n)\in \calP(m_1-2, m_2, \ldots, m_n)$ is a general pointed curve such that $z_1$ is not a Weierstrass point and that $m_1\nmid g$. Then $z_1$ in a general curve $(C, z_1, \ldots, z_n)\in \calP(m_1, \ldots, m_n)$ is not a Weierstrass point. 
\end{proposition}

\begin{proof}
By assumption, we can take a general curve $(X, q, z_2, \ldots, z_n) \in \calP(m_1 -2, m_2, \ldots, m_n)$ such that $q$ is not a Weierstrass point in $X$. Attach to $X$ at $q$ an elliptic curve $E$ with an $m_1$-torsion point $z_1$ to $q$, i.e. $m_1 z_1 \sim m_1 q$ in $E$, such that 
$gz_1\not\sim gq$, see Figure~\ref{fig:weier-1}. 

\begin{figure}[h]
    \centering
    \psfrag{z2}{$z_2$}
    \psfrag{zn}{$z_n$}
    \psfrag{z1}{$z_1$}
     \psfrag{X}{$X$}
    \psfrag{q}{$q$}
     \psfrag{E}{$E$}
    \includegraphics[scale=0.8]{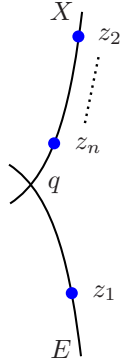}
    \caption{\label{fig:weier-1} A curve union an elliptic tail marked at an $m_1$-torsion point}
\end{figure}

Such $z_1$ in $E$ exists because $m_1\nmid g$. By Theorem~\ref{thm:canonical}, 
the union of $X$ and $E$ with the marked points $z_i$'s is contained in $\BPP(m_1, \ldots, m_n)$. 
Since $gz_1\not\sim gq$ in $E$, it implies that $z_1$ is not a limit of Weierstrass points degenerating from smooth curves to $X$ union $E$ (see \cite[Theorem 5.45]{HarrisMorrison}). It follows that $z_1$ in a general curve $(C, z_1, \ldots, z_n)\in \calP(m_1, \ldots, m_n)$ is not a Weierstrass point. 

One subtlety is when $\calP(m_1, m_2, \ldots, m_n)$ has two spin components (we have already discussed hyperelliptic components). In this case $\calP(m_1 -2, m_2, \ldots, m_n)$ also has two spin components. Since $X$ union $E$ 
is of compact type, the spin parity still holds by Theorem~\ref{thm:spin}, and is determined by that of each component. Therefore, we can take a general $X$ in each of the
spin components of $\calP(m_1 -2, m_2, \ldots, m_n)$ and carry out the above smoothing argument, which will lead to each of the spin components of $\calP(m_1, m_2, \ldots, m_n)$ (see also \cite{Bullock} for a similar discussion). 
\end{proof}

\begin{proposition}
\label{prop:weierstrass-meromorphic}
Let $(X, z_1, \ldots, z_r, p_1, \ldots, p_s)\in \calP(k_1, \ldots, k_r, -l_1 - 2, - l_2, \ldots, -l_s)$ be a general curve. 
Suppose that $p_1$ is not a Weierstrass point and that $l_1\nmid g$. Then $p_1$ in 
a general curve $(C, z_1, \ldots z_r, p_1, \ldots, p_s)\in \calP(k_1, \ldots, k_r, -l_1, \cdots, -l_s)$ is not a Weierstrass point. 
\end{proposition}

\begin{proof}
Take general
$(X, z_1, \ldots z_r, q, p_2, \ldots, p_s) \in \calP(k_1, \ldots, k_r, -l_1 - 2, - l_2, \ldots, -l_s)$ such that $q$ is not a Weierstrass point. Attach to $X$ at $q$ an elliptic curve $E$ with an $l_1$-torsion point $p_1$ to $q$, i.e. $l_1 p_1 \sim l_1 q$ in $E$, such that $gp_1 \not\sim gq$, see Figure~\ref{fig:weier-2}. 

\begin{figure}[h]
    \centering
    \psfrag{z1}{$z_1$}
    \psfrag{zr}{$z_r$}
    \psfrag{p1}{$p_1$}
     \psfrag{p2}{$p_2$}
      \psfrag{ps}{$p_s$}
     \psfrag{X}{$X$}
    \psfrag{q}{$q$}
     \psfrag{E}{$E$}
    \includegraphics[scale=0.8]{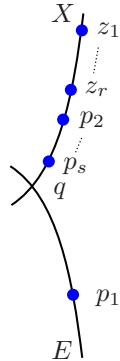}
    \caption{\label{fig:weier-2} A curve union an elliptic tail marked at an $l_1$-torsion point}
\end{figure}

Such $p_1$ exists because $l_1\nmid g$. 
By Theorem~\ref{thm:twisted-meromorphic}, the union of $X$ and $E$ with the marked points $z_i$'s and $p_j$'s is contained in $\BPP(k_1, \ldots, k_r, -l_1, \cdots, -l_s)$. Since $g p_1\not\sim gq$ in $E$, it 
 implies that $p_1$ is not a limit of Weierstrass points degenerating from smooth curves to $X$ union $E$, thus proving the desired claim. The case when the stratum consists of spin components can be treated by the same argument as in the proof of Proposition~\ref{prop:weierstrass-holomorphic}. 
 \end{proof}

Let us prove Theorem~\ref{thm:weierstrass}. For the reader's convenience, we recall its content as follows. 

\begin{theorem}
\label{thm:weierstrass-zero-pole}
Let $(C, z_1, \ldots, z_r, p_1, \ldots, p_s)\in \calP(k_1, \ldots, k_r, -l_1, \ldots, -l_s)$ (non-hyperelliptic component) be a general curve. 
Then $z_1$ is not a Weierstrass point. 
\end{theorem}

\begin{proof}
Let $k = \sum_{i=1}^r$ and $l = \sum_{j=1}^s$. Then $2g-2 = k - l$. Analogous to Lemmas~\ref{lem:weierstrass-two} and~\ref{lem:weierstrass-two-meromorphic}, it suffices to prove the result for $\calP(k, -l)$. Do induction on $g$. The case $g=1$ is trivial, because there is no Weierstrass point on an elliptic curve. Suppose the claim holds for $g-1$. Take $(X, q, p) \in \calP(k-2, -l)$ such that $q$ is not a Weierstrass point. Attach to $X$ an elliptic curve $E$ at $q$. Take a $k$-torsion $z$ with respect to $q$ in $E$ such that $z$ and $q$ are not $g$-torsion to each other. This is feasible because $k = 2g-2+l > g$. By Theorem~\ref{thm:twisted-meromorphic} we have $(X\cup_q E, z, p) \in \BPP(k, -l)$, and $z$ is not a limit Weierstrass point, thus proving the desired result by induction. Again, the case when the stratum consists of spin components can be treated by the same argument as in Proposition~\ref{prop:weierstrass-holomorphic}. 
\end{proof}

\subsection{Chains of elliptic curves}
\label{subsec:chain}

For the sake of completeness, let us analyze limits of Weierstrass points as smooth curves degenerate to 
a chain $C$ of $g$ elliptic curves $E_1, \ldots, E_g$. Suppose $E_i \cap E_{i+1} = q_i$ for $i=1,\ldots, g-1$. Let $q_g \in E_g$ be a smooth point, see Figure~\ref{fig:weier-elliptic-chain}. 

 \begin{figure}[h]
    \centering
   \psfrag{E1}{$E_1$}
    \psfrag{E2}{$E_2$}
    \psfrag{Eg-1}{$E_{g-1}$}
     \psfrag{Eg}{$E_g$}
      \psfrag{z}{$q_g$}
    \psfrag{q1}{$q_1$}
     \psfrag{qg-1}{$q_{g-1}$}
    \includegraphics[scale=1.0]{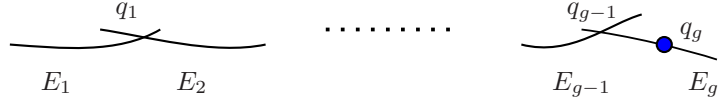}
       \caption{\label{fig:weier-elliptic-chain} A chain of elliptic curves marked in the last component} 
 \end{figure}      

Denote by $t(x, y)$ the \emph{torsion order} of $x$ and $y$ in an elliptic curve, i.e. the minimal positive integer $t$ such that $t x \sim t y$. If such $t$ does not exist, set $t(x, y) = \infty$. In particular, let $t_i = t (q_{i-1}, q_i)$ in $E_i$ for $i=2,\ldots, g$. 

\begin{proposition}
\label{prop:weierstrass-elliptic-chain}
In the above setting, $q_g$ is a limit Weierstrass point if and only if there exists an integer sequence 
$g = k_g \geq k_{g-1} \geq \cdots \geq k_2 \geq k_1 \geq 2$ such that  
$k_{i} q_i - k_{i-1} q_{i-1}$ is an effective divisor class in $E_i$ for $i = 2, \ldots, g$.  
\end{proposition}

\begin{proof}
Note that $q_g$ is a limit Weierstrass point if and only if there exists an admissible cover of degree $g$, totally ramified at $q_g$ as in Figure~\ref{fig:weier-cover}. 
\begin{figure}[h]
    \centering
    \psfrag{Eg}{$E_g$}
    \psfrag{Eg-1}{$E_{g-1}$}
     \psfrag{Eg-2}{$E_{g-2}$}
      \psfrag{qg}{$q_g$}
       \psfrag{g}{$g$ sheets}
     \psfrag{qg-1}{$q_{g-1}$}
    \psfrag{qg-2}{$q_{g-2}$}
    \includegraphics[scale=0.7]{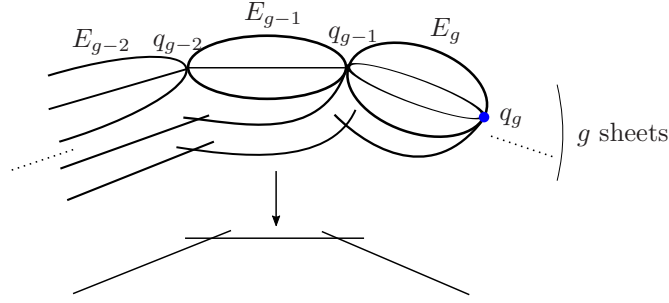}
    \caption{\label{fig:weier-cover} An admissible cover totally ramified at $q_g$}
\end{figure}

Suppose that $k_i$ encodes the ramification order of the cover at $q_i \in E_{i+1}$. Start from $E_g$ and keep track of the cover along the chain at each node $q_{g-1}, \ldots, q_1$. The desired result follows right away. 
\end{proof}

\begin{remark}
If $k_i > k_{i-1}$, the effectiveness of $k_{i} q_i - k_{i-1} q_{i-1}$ always holds by the group structure of elliptic curves. On the other hand if 
$k_i = k_{i-1}$, then the effectiveness condition implies that $t_i \mid k_i$. 
\end{remark}

\begin{corollary}
\label{cor:weierstrass-elliptic-chain}
In the above setting, if $t_i \mid i$ for some $i$, then $q_g$ is a limit Weierstrass point. 
\end{corollary}

\begin{proof}
Take $k_j = j+1$ for $j\leq i-1$ and $k_l = l$ for $l \geq i$. 
Since $i q_i \sim i q_{i-1}$, the result follows from Proposition~\ref{prop:weierstrass-elliptic-chain}. 
\end{proof}

\section{Boundary of the minimal stratum in genus three}
\label{sec:h(4)}

In genus three, the minimal stratum $\calP(4)$ consists of two connected components $\calP(4)^{\hyp}$ and $\calP(4)^{\odd}$. Denote by $z$ the unique zero. We would like to classify which stable nodal curve $(C, z)\in \BM_{3,1}$ is contained in $\BPP(4)^{\hyp}$ and in $\BPP(4)^{\odd}$. As we have seen before, using admissible double covers provides us a good understanding of $\BPP(4)^{\hyp}$. Moreover, in this case $\calP(4)^{\even}$ coincides 
with $\calP(4)^{\hyp}$, hence the result in Theorem~\ref{thm:spin} also applies. Below we will classify boundary points of $\BPP(4)^{\hyp}$ and $\BPP(4)^{\odd}$ for curves with at most two nodes. We separate the discussion by the number of nodes and the topological type of a curve. 

\subsection{Curves with one node}

There are three cases for $(C, z)\in \BM_{3,1}$ such that $C$ has one node, see Figure~\ref{fig:g3-onenode}. 
\begin{figure}[h]
    \centering
    \psfrag{z}{$z$}
    \psfrag{q}{$q$}
    \psfrag{A}{$C: g=2$}
    \psfrag{B}{$C_1: g=1$}
    \psfrag{C}{$C_2: g=2$}
    \includegraphics[scale=0.9]{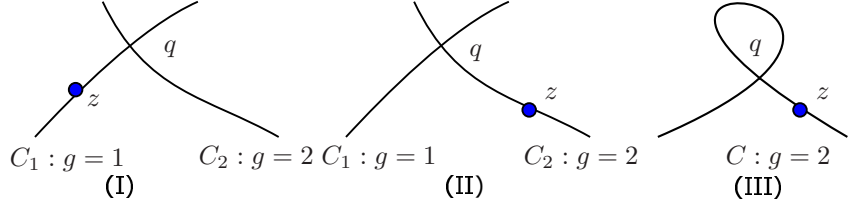}
    \caption{\label{fig:g3-onenode} Stable pointed genus three curves with one node}
\end{figure}

Here we denote by $q$ the node, $z$ the marked point, and label the \emph{geometric genus} of each component of $C$. 

\begin{itemize}
\item Case (I). By Theorem~\ref{thm:canonical}, $(C, z) \in \BPP(4)$ if and only if $4z\sim 4q$ in $C_1$ and $2q\sim K_{C_2}$ in $C_2$.  
In particular, the condition on $C_2$ implies that $q$ is a Weierstrass point of $C_2$. It suffices to distinguish which component of  $\BPP(4)$ contains
$(C,z)$. As in the proof of Theorem~\ref{thm:spin}, the limit spin structure on $C$ is $(\OO_{C_1}(2z-2q), \OO_{C_2}(q))$ (we drop the exceptional $\bbP^1$), hence the parity is even if and only if $2z\sim 2q$ in $C_1$. We thus conclude that $(C, z) \in \BPP(4)^{\hyp}$ if and only if $2z\sim 2q$ in $C_1$ and $2q\sim K_{C_2}$ in $C_2$, and $(C, z) \in \BPP(4)^{\odd}$ if and only if $4z\sim 4q$, 
$2z\not\sim 2q$, and $2q\sim K_{C_2}$. 

\item Case (II). As above, $(C, z) \in \BPP(4)$ if and only if $4z\sim 2q + K_{C_2}$ in $C_2$. The limit spin structure on $C$ is 
$(\OO_{C_1}, \OO_{C_2}(2z-q))$. Its parity 
is even if and only if $2z-q$ is effective in $C_2$. Note that $2z-q$ is effective if and only if $z$ is a Weierstrass point of $C_2$, i.e. $2z\sim K_{C_2}$, which further implies that $2z\sim 2q$. We thus conclude that $(C, z) \in \BPP(4)^{\hyp}$ if and only if 
$2z\sim 2q\sim K_{C_2}$, and $(C, z) \in \BPP(4)^{\odd}$ if and only if $4z\sim 2q + K_{C_2}$ and $2z\not\sim K_{C_2}$. 

\item Case (III). Let $\wt{C}$ be the normalization of $C$, and $q_1, q_2\in \wt{C}$ the preimages of $q$. A stable one-form on $\wt{C}$ with a zero of multiplicity four at $z$ corresponds to a section of $K_{\wt{C}}(q_1 + q_2)$. By Theorem~\ref{thm:hodge}, $(C, z)\in \BPP(4)$ if and only 
if $4z\in K_{\wt{C}} + q_1 + q_2$ in $\wt{C}$. The limit spin structure on $C$ is given by $\OO_C(2z)$. Hence it is even if and only if 
$C$ is a (degenerate) hyperelliptic curve and $z$ is a ramification point of the corresponding admissible double cover, i.e. if and only if 
$2z\sim q_1 + q_2\sim K_{\wt{C}}$. We thus conclude that $(C, z) \in \BPP(4)^{\hyp}$ if and only if 
$q_1 + q_2\sim 2z\sim K_{\wt{C}}$, and $(C, z) \in \BPP(4)^{\odd}$ if and only if $4z\in K_{\wt{C}} + q_1 + q_2$ but $2z\not \sim K_{\wt{C}}$.  \end{itemize}

\subsection{Curves with two nodes}

There are ten cases for $(C, z)\in \BM_{3,1}$ such that $C$ has two nodes, see Figure~\ref{fig:g3-twonode}.  
\begin{figure}[h]
    \centering
    \psfrag{z}{$z$}
    \psfrag{q1}{$q_1$}
    \psfrag{q2}{$q_2$}
    \psfrag{a}{$C: g=1$}
    \psfrag{b}{$C_0: g=0$}
     \psfrag{c}{$C_1: g=1$}
    \psfrag{d}{$C_2: g=2$}
    \psfrag{e}{$C_2: g=1$}
    \psfrag{f}{$C_3: g=1$}
    \psfrag{g}{$C_1: g=0$}
    \includegraphics[scale=0.85]{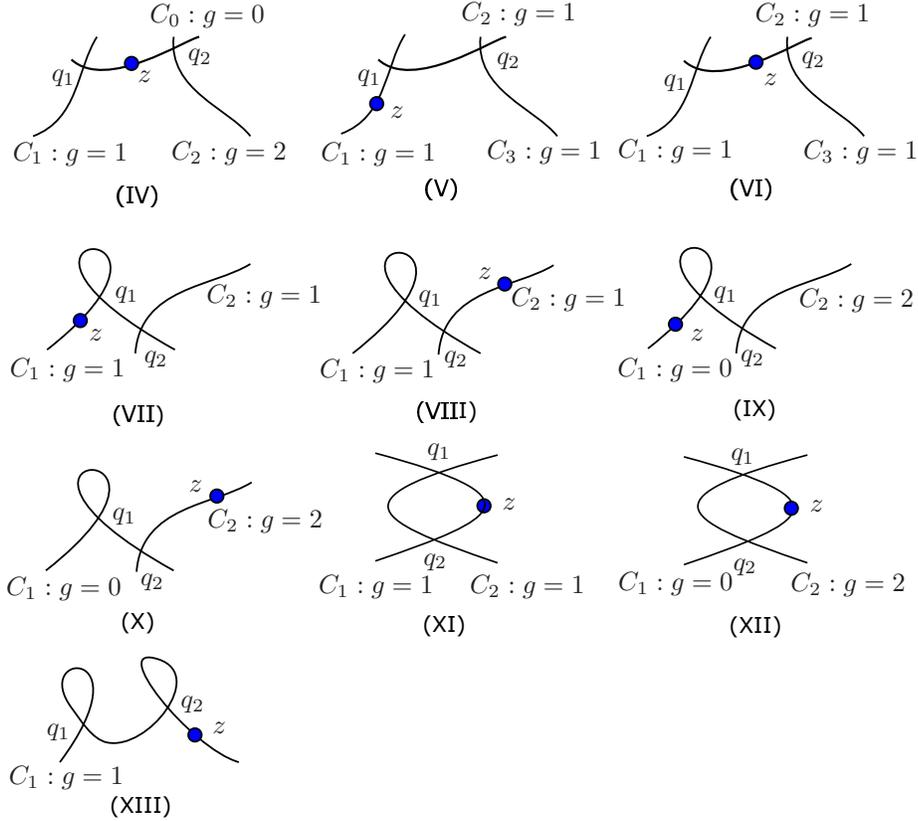}
    \caption{\label{fig:g3-twonode} Stable pointed genus three curves with two nodes}
\end{figure}

Again, we denote by $q_1, q_2$ the two nodes, $z$ the marked point, and label the \emph{geometric genus} of each component of $C$

\begin{itemize}
\item Case (IV). Note that $C$ is of compact type. If $(C, z)\in \BPP(4)$, 
the limit spin structure on $C$ would be $(\OO_{C_0}(-1), \OO_{C_1}, \OO_{C_2}(q_2))$, hence its parity 
is even, which implies that $(C, z)\in \BPP(4)^{\hyp}$. Nevertheless, by analyzing possible admissible double covers $f$ on $C$, it follows that 
$f|_{C_0}$ has ramification at $q_1, q_2$ and $z$, contradicting the fact that on $C_0 \cong \bbP^1$ there are only two ramification points
by Riemann-Hurwitz. We thus conclude that $\BPP(4)$ is disjoint from the locus of curves of type (IV).

\item Case (V). If $(C, z)\in \BPP(4)$, necessarily we have $4z\sim 4q_1$ in $C_1$ and $2q_1\sim 2q_2$ in $C_2$ by Proposition~\ref{prop:limit-canonical}. In addition, 
the limit spin structure on $C$ is $(\OO_{C_1}(2z-2q_1), \OO_{C_2}(q_1 - q_2), \OO_{C_3})$. It is even if and only if 
$2z \sim 2q_1$. In this case, again using admissible covers we see that $(C, z)\in \BPP(4)^{\hyp}$ if and only if 
$2z\sim 2q_1$ in $C_1$ and $2q_1 \sim 2q_2$ in $C_2$. Moreover, $(C, z)\in \BPP(4)^{\odd}$ if and only if 
$2q_1 \sim 2q_2$ in $C_2$, $4z\sim 4q_1$ but $2z\not\sim 2q_1$ in $C_1$, where the ``if'' part follows from 
Theorem~\ref{thm:canonical-more} because $C_3$ is the only holomorphic component. 

\item Case (VI). If $(C, z)\in \BPP(4)$, necessarily we have $4z\sim 2z_1 + 2z_2$ in $C_2$. The limit spin structure on $C$ 
is $(\OO_{C_1}, \OO_{C_2}(2z- q_1 - q_2), \OO_{C_3})$, hence it is odd if and only if $2z\sim q_1 + q_2$ in $C_2$. Using admissible covers we see that $(C, z)\in \BPP(4)^{\hyp}$ if and only if $2z\sim 2q_1 \sim 2q_2$ in $C_2$. On the other hand, 
if $2z\sim q_1 + q_2$, then certainly $2z\not\sim 2q_i$ for $i=1,2$. Furthermore, if there exists a limit $g^2_4$ on $C$ such that 
the vanishing sequences of its aspect on $C_2$ at $z$, $q_1$ and $q_2$ are $(0,1,4)$, $(0,2,3)$ and $(0,2,3)$, respectively, such $g^2_4$ imposes 
two conditions to pointed curves of type (VI), matching the adjusted Brill-Noether number by imposing vanishing sequence $(0,1,4)$ to $z$. Hence 
by the smoothability criterion of limit linear series (\cite[Theorem 3.4]{EisenbudHarrisLimit}) such $(C, z)$ appears in $\BPP(4)^{\odd}$. 

\item Case (VII). Note that $C$ is of pseudocompact type, so the discussion of twisted canonical divisors and limit spin structures also applies. The situation is very similar to Case (II). Let $q'_1$ and $q_1''$ be the preimages of $q_1$ in $\wt{C}_1$. Then $(C, z) \in \BPP(4)^{\hyp}$ if and only if 
$2z\sim 2q_2\sim q_1' + q_1''$ in $\wt{C}_1$, again easily seen by admissible covers. On the other hand, 
$(C, z) \in \BPP(4)^{\odd}$ if and only if $4z\sim 2q_2 + q_1' + q_1''$ and $2z\not\sim 2q_2$, where the ``if'' part follows from 
plumbing a cylinder at $q_1$ in the proof of Theorem~\ref{thm:hodge} and applying Theorem~\ref{thm:canonical}. 

\item Case (VIII). This is similar to Case (I). We conclude that $(C, z) \in \BPP(4)^{\hyp}$ if and only if $2q_2\sim q_1' + q_1''$ in $C_1$ and $2z\sim 2q_2$ in $\wt{C}_2$ by admissible covers, and $(C, z) \in \BPP(4)^{\odd}$ if and only if $2q_2\sim q_1' + q_1''$ in $\wt{C}_1$, $4z\sim 4q_2$ and $2z\not\sim 2q_2$ in $C_2$. 

\item Case (IX). This is similar to Case (I). We conclude that $(C, z) \in \BPP(4)^{\hyp}$ if and only if $2q_2\sim K_{C_2}$ and $q_1', q_1''$ are conjugate under the double cover induced by $2z\sim 2q_2$ on $\wt{C}_1$. On the other hand, $(C, z) \in \BPP(4)^{\odd}$ if and only if $2q_2\sim K_{C_2}$ and $z, q_2$ are primitive $4$-torsions in the rational nodal curve $C_1$. 

\item Case (X). This is similar to Case (II). We conclude that $(C, z) \in \BPP(4)^{\hyp}$ if and only if $2z\sim 2q_2$ in $C_2$, 
and $(C, z) \in \BPP(4)^{\odd}$ if and only if $4z\sim 2q_2 + K_{C_2}$ and $2z\not\sim K_{C_2}$. 

\item Case (XI). Suppose $(C, z)\in \BPP(4)$. The 
limit spin structure is $(\OO_{C_1}(2z-q_1-q_2), \OO_{\bbP^1}(1), \OO_{\bbP^1}(1), \OO_{C_2})$, after blowing up $q_1, q_2$ and inserting two exceptional $\bbP^1$-components. The parity is determined by 
$h^0(C_1, 2z-q_1 - q_2) + h^0(C_2, \OO)$, hence it is even if and only if $2z\sim q_1 + q_2$ in $C_1$, i.e. if and only if 
$(C_1, z)\in \BPP(4)^{\hyp}$ by using admissible covers. Now suppose $(C, z)\in \BPP(4)^{\odd}$ and $2z\not\sim q_1 + q_2$. 
We claim that $4z\sim 2q_1 + 2q_2$ in $C_1$, which can be seen as follows. Embed 
$C_1$ to $\bbP^3$ by the linear system $|2q_1 + 2q_2|$ as an elliptic normal quartic. Let $H$ be the plane in $\bbP^3$ that cuts out $2q_1 + 2q_2$ 
in $C_1$. Choose a point in $H$ and project $C_1$ from it to $\bbP^2$. The image of $C_1$ is a plane quartic $C'_1$ with a tacnode $q$, and $H$ maps to a line that passes through $q$ and is tangent to both branches of $C'_1$ at $q$. When smooth plane quartics with a hyperflex degenerate to $C'_1$, the limit hyperflex line cuts out $4z$, hence we conclude that $4z\sim 2q_1 + 2q_2$. Conversely if $4z\sim 2q_1 + 2q_2$ in $C_1$, 
since the space of plane quartics with a hyperflex is irreducible, the tacnodal elliptic quartic model $C'_1$ of $C_1$ appears as a limit 
of smooth plane quartics with a hyperflex, with $z$ as the limit hyperflex point. Running stable reduction to a general pencil in this degeneration to resolve the tacnode, the stable limit is of type (XI) and all possible elliptic bridges $C_2$ show up. Moreover, in this case we see that $2z\not\sim q_1 + q_2$. Otherwise $2z + q_1 + q_2 \sim 2q_1 + 2q_2$ would imply that the tangent line to $C'_1$ at $z$ cuts out $2z+q_1 + q_2$, contradicting that 
$z$ is a hyperflex of $C'_1$. Hence it belongs to $\BPP(4)^{\odd}$, and not to $\BPP(4)^{\hyp}$. 

\item Case (XII). By admissible covers we see that $(C, z)\in \BPP(4)^{\hyp}$ if and only if $q_1 + q_2 \sim K_{C_2}$. In this case it also belongs 
to $\BPP(4)^{\odd}$ as proved in Theorem~\ref{thm:double-conic}. Suppose now 
$(C, z)\in \BPP(4)^{\odd}$ and $C$ is not hyperelliptic, i.e. $q_1 + q_2 \not\sim K_{C_2}$. Let $C'$ be the irreducible nodal curve by blowing down $C_1$, i.e. identifying $q_1$ and $q_2$ in $C_2$ as a node $q$. Since $C'$ is not hyperelliptic, it admits a canonical embedding as a plane nodal quartic. Consider smooth curves in $\calP(4)^{\odd}$ as plane quartics with a hyperflex line. When they degenerate to $C'$, in order to have contact order four to $C'$ at $q$, the limit $L$ of hyperflex lines cuts out either $3q_1 + q_2$ or $q_1 + 3q_2$ in $C_2$. Without loss of generality, suppose the former occurs. It then implies that $3q_1 + q_2 \sim K_{C_2}(q_1 + q_2)$, hence $q_1$ is a Weierstrass point of $C_2$. Conversely if 
$2q_1 \sim K_{C_2}$ (or $2q_2\sim K_{C_2}$ by symmetry), since the space of plane quartics that have contact order four to $L$ at $q$ is irreducible, such a $(C,z)$ appears as a limit of smooth plane quartics with a hyperflex, hence it is contained in $\BPP(4)^{\odd}$. 

\item Case (XIII). Using admissible covers, we obtain that $(C, z) \in \BPP(4)^{\hyp}$ if and only if 
$q'_1 + q''_1 \sim q'_2 + q''_2\sim 2z$ in $\wt{C}$. On the other hand, $(C, z) \in \BPP(4)^{\odd}$ if and only if 
$4z\sim q'_1 + q''_1 + q'_2 + q''_2$, the corresponding meromorphic differential in $\wt{C}$ has residues summing up to zero at $q'_i, q_i''$ for both $i=1,2$, and $2z\not\sim q_i' + q_i''$ for either $i$, where the ``if'' part follows from plumbing a cylinder at $q_1$ and at $q_2$.  
\end{itemize}

It appears possible to classifying curves with more than two nodes that are contained in $\BPP(4)$. But the number of topological types of curves with three nodes or more increases significantly, so here we choose to end our discussion. 


\end{document}